\documentclass[10pt]{amsart}
\usepackage{amsmath}
\usepackage{amsfonts}
\usepackage{amssymb}
\usepackage{amsthm}
\usepackage{url}
\usepackage[all]{xy}
\usepackage{dsfont}
\usepackage{graphicx}
\usepackage{caption}
\usepackage{subcaption}
\usepackage{comment} 
\usepackage{stmaryrd}
\usepackage{hyperref}
\usepackage{todonotes}
\usepackage{color}
\usepackage{enumerate,tikz-cd}
\usepackage{fixltx2e}
\usepackage{MnSymbol}
\allowdisplaybreaks

\numberwithin{equation}{section}

\newtheorem{proposition}{Proposition}[section]
\newtheorem{lemma}[proposition]{Lemma}
\newtheorem{theorem}[proposition]{Theorem}
\newtheorem{corollary}[proposition]{Corollary}
\newtheorem{conjecture}{Conjecture}[section]

\theoremstyle{definition}
\newtheorem{remark}[proposition]{Remark}
\newtheorem{definition}[proposition]{Definition}
\newtheorem{example}[proposition]{Example}

\DeclareMathOperator{\Aut}{Aut}

\DeclareMathOperator{\DF}{DF}

\DeclareMathOperator{\lct}{lct}
\DeclareMathOperator{\Ding}{Ding}
\DeclareMathOperator{\Weil}{Weil}
\DeclareMathOperator{\divisor}{div}
\DeclareMathOperator{\can}{can}
\DeclareMathOperator{\ord}{ord}

\newcommand{\R}{\mathbb{R}}
\newcommand{\C}{\mathbb{C}}

\newcommand{\Q}{\mathbb{Q}}

\newcommand{\pr}{\mathbb{P}}
\renewcommand{\epsilon}{\varepsilon}

\newcommand{\D}{\mathcal{D}}

\newcommand{\scA}{\mathcal{A}}
\newcommand{\I}{\mathcal{I}}
\newcommand{\scO}{\mathcal{O}}

\newcommand{\J}{\mathcal{J}}

\newcommand{\mfa}{\mathfrak{a}}

\newcommand{\E}{\mathcal{E}}
\renewcommand{\L}{\mathcal{L}}
\newcommand{\X}{\mathcal{X}}
\newcommand{\Y}{\mathcal{Y}}

\renewcommand{\phi}{\varphi}
\newcommand{\ddc}{dd^c}

\newcommand\PSH{\mathrm{PSH}}
\newcommand\MA{\mathrm{MA}}

\newcommand\mi{^{-1}}
\newcommand\vol{\mathrm{vol}}

\newcommand\bsni{\bigskip\noindent}

\newcommand\bbc{\mathbb{C}}
\newcommand\bbd{\mathbb{D}}

\newcommand\bbp{\mathbb{P}}
\newcommand\bbq{\mathbb{Q}}
\newcommand\bbr{\mathbb{R}}

\newcommand\cA{\mathcal{A}}

\newcommand\cD{\mathcal{D}}
\newcommand\cE{\mathcal{E}}

\newcommand\cL{\mathcal{L}}

\newcommand\cO{\mathcal{O}}

\newcommand\cX{\mathcal{X}}
\newcommand\cY{\mathcal{Y}}

\newcommand\triv{{\mathrm{triv}}}

\title[Ding stability and K\"ahler--Einstein metrics]{ Ding stability and K\"ahler--Einstein metrics on manifolds with big anticanonical class}

\author[Ruadha\'i Dervan and R\'emi Reboulet]{Ruadha\'i Dervan and R\'emi Reboulet}

\address{Ruadha\'i Dervan, School of Mathematics and Statistics, University of Glasgow, University Place, Glasgow G12 8QQ, United Kingdom}\email{ruadhai.dervan@glasgow.ac.uk}

\address{R\'emi Reboulet, DPMMS, Centre for Mathematical Sciences, Wilberforce Road, Cambridge CB3 0WB, United Kingdom}\email{rr653@cam.ac.uk, rebouletremimath@gmail.com}
\begin{document}

\begin{abstract}

We introduce a notion of uniform Ding stability for a projective manifold with big anticanonical class, and prove that the existence of a unique K\"ahler--Einstein metric on such a manifold implies uniform Ding stability. The main new techniques are to develop a general theory of Deligne functionals---and corresponding slope formulas---for singular metrics, and to prove a slope formula for the Ding functional in the big setting. This extends work of Berman in the Fano situation,   when the anticanonical class is actually ample, and proves one direction of the analogue of the Yau--Tian--Donaldson conjecture in this setting. We also speculate about the relevance of uniform Ding stability and K-stability to moduli in the big setting.

\end{abstract}


\maketitle

\tableofcontents

\section{Introduction}

The Yau--Tian--Donaldson conjecture states that the existence of a K\"ahler--Einstein metric on a Fano manifold $X$ should be equivalent to the purely algebro-geometric notion of \emph{K-stability} of $(X,-K_X)$ \cite{STY,GT,SD1}, where $X$ being Fano means that the anticanonical class $-K_X$ is ample. There are, by now, many different proofs of this conjecture \cite{datszek,CSW,KZ,bbj,chiliguniform}, following the initial work of Chen--Donaldson--Sun \cite{CDS} (who proved K-stability implies existence) and Tian and Berman (who proved existence implies K-stablity) \cite{GT, bermankpoly}. The notion of K-stability involves associating a numerical invariant---called the \emph{Donaldson--Futaki invariant}---to certain algebro-geometric degenerations of $(X,-K_X)$ called \emph{test configurations}; K-stability asks that the Donaldson-Futaki invariant is always positive.

The alternative proof that we advertise, which applies when $\Aut(X)$ is discrete and proves a slightly weaker statement, is due to Berman--Boucksom--Jonsson and roughly goes as follows \cite{bbj}. Firstly, the existence of a K\"ahler--Einstein metric is equivalent to coercivity of the \emph{Ding functional} \cite{bob, bbgz}, a functional on the space of K\"ahler metrics lying in $c_1(X)=c_1(-K_X)$. Assuming the Ding functional is not coercive, they produce a (weak) ``destabilising'' geodesic ray in the space of K\"ahler metrics and approximate it by geodesic rays arising from test configurations. Rather than K-stability, the key algebraic input is instead the notion of \emph{uniform Ding stability}. Important work of Berman demonstrates that, given a test configuration, the slope of the Ding functional along an associated geodesic is the \emph{Ding invariant} of a test configuration  \cite{bermankpoly}, which has well-behaved enough approximation properties to  construct a destabilising test configuration from the destabilising geodesic.  Finally, input from birational geometry shows that uniform Ding stability is actually equivalent to uniform K-stability. We emphasise that it is really uniform Ding stability that plays the crucial role in their proof, and part of the reason for this is that Ding stability is more suited than K-stability to situations in which less regularity is available.

The goal of this paper is develop a theory of Ding stability when $-K_X$ is \emph{big} rather than \emph{ample}, and to relate Ding stability to the existence of K\"ahler--Einstein metrics in this generality. The condition that a line bundle $L$ on $X$ be big asks that the \emph{volume} $$\vol(L) = \lim_{k\to\infty}\frac{\dim H^0(X,kL)}{k^n/n!}$$ is strictly positive; this concept plays a crucial role in many different aspects of modern algebraic geometry. We highlight two such instances: the first is that the condition that $K_X$ be big means that $X$ is of \emph{general type}, and these are the primary setting of birational geometry \cite{BCHM}. The second is recent work of Li \cite{CL}, who explains how to use big line bundles to study the existence of constant scalar curvature K\"ahler metrics in the first Chern class of an ample line bundle $L$---here big line bundles are extremely useful even for applications to the ample setting.

The study of K\"ahler--Einstein metrics in the big setting began with early work of Tsuji \cite{tsuji}, and on manifolds with $K_X$ big a complete theory now exists, particularly due to work of Eyssidieux--Guedj--Zeriahi \cite{EGZ} (see also \cite{begz, bbgz} for further developments). The ``opposite'' situation, namely when $-K_X$ is big, is the one of interest to us, and its connection with the Yau-Tian-Donaldson conjecture seems to have not been considered before. As in the Fano setting with $-K_X$ ample, such metrics cannot exist in general, and one expects that their existence should be intimately related with algebro-geometric stability. Our main result confirms this:

\begin{theorem}\label{intromainthm} Let $X$ be a smooth projective variety with $-K_X$ big.
\begin{enumerate}[(i)]
\item If $X$ admits a K\"ahler--Einstein metric, then it is  Ding semistable.
\item If $X$ admits a unique K\"ahler--Einstein metric, then it is uniformly Ding stable.
\end{enumerate}
\end{theorem}

We expect that uniqueness of K\"ahler--Einstein metrics happens precisely when $\Aut(X)$ is discrete. In particular, our result proves one direction of the Yau--Tian--Donaldson in this setting.

Technically, the main difficulty with K\"ahler geometry in big classes is that the most natural metrics are actually \emph{singular}, and hence one must employ the tools of pluripotential theory. The least singular metrics are those with \emph{minimal singularities}, and  even defining K\"ahler--Einstein condition in this setting requires some care: we must assume that metrics with minimal singularities are \textit{klt}, i.e. that their multiplier ideal sheaf is trivial. This condition can be formulated purely algebraically in our setting, by asking for the asymptotic multiplier ideal sheaf of the linear series $\|-K_X\|$ to be trivial. Thus a first step towards Theorem \ref{intromainthm} is to develop the foundations of the theory of K\"ahler--Einstein metrics and the Ding functional in the big setting. Our setting is morally related to recent work of Trusiani \cite{AT}, who develops a theory of K\"ahler--Einstein metrics with \emph{prescribed singularities}; Trusiani assumes $-K_X$ ample, but fixes the singularity type of the metrics of interest. By adapting his ideas, using foundational pluripotential-theoretic work of Darvas--di Nezza--Lu \cite{ddnlmetric, ddnlfullmass} and a key convexity result of Berndtsson--P\u{a}un \cite{bpaun}, we prove the following:

\begin{theorem}\label{introanalytichm} Let $X$ be a  smooth projective variety with $-K_X$ big.
\begin{enumerate}[(i)]
\item If $X$ admits a K\"ahler--Einstein metric, then  the Ding functional is bounded below.
\item If $X$ admits a unique K\"ahler--Einstein metric, then the Ding functional is coercive.
\end{enumerate}
\end{theorem}

The technical core of the paper is then to relate coercivity of the Ding functional to uniform Ding stability. To explain this, we first briefly explain our notion of a test configuration, which makes sense for a general big line bundle $L$ on $X$ (which in applications to the existence of K\"ahler--Einstein metrics will be taken to be $-K_X$). 

In the ample case, the definition of a test configuration involves a $\C^*$-degeneration $\X\to \pr^1$, and a $\Q$-line bundle $\L$ over $\X$, with $(\cX,\cL)$ restricting to $(X,L)$ on the general fibre \cite{SD1}; $\cL$ can be taken to be ample. Our notion of a test configuration differs from this in two ways. Firstly, the theory of big line bundles is birational in flavour, and so we allow the general fibre to not only be $(X,L)$, but also certain \emph{birational models} of $(X,L)$ in a sense made precise below; this is crucial both conceptually and for applications to algebraic geometry. Secondly, we ask that the line bundle on the total space of the test configuration be big, rather than ample. We emphasise that in this way the set of test configurations for $(X,-K_X)$ is naturally associated to the \emph{birational equivalence class} of $(X,-K_X)$.

Relative to a fixed smooth reference form $\theta \in c_1(X)$, the Ding functional is a difference of two components: the first is the \emph{Monge--Amp\`ere energy}, sending a psh function $\phi \to E(\phi)$; this functional appears in  many different aspects of K\"ahler geometry (having been introduced in \cite{bb} in the big setting). The second, $\phi \to L(\phi)$, is more specific to the K\"ahler--Einstein problem. Similarly our Ding invariant associated to a test configuration is a difference of two terms, the first being the \emph{volume} of $\L$---which appears in many different aspects of algebraic geometry---while the second is again more specific to the fact that the line bundle of interest is $-K_X$. With this definition in hand, Ding semistability asks that the Ding invariant be non-negative for all test configurations, with uniform Ding stability defined in a similar way; these again are properties of the birational equivalence class of $(X,-K_X)$.

Our slope formulas---which are the key to proving Theorem \ref{intromainthm} from Theorem \ref{introanalytichm}---begin by proving a slope formula for the Monge--Amp\`ere energy for paths of singular metrics; for rays of metrics associated to test configurations, we show that the slope of the Monge--Amp\`ere energy is the \emph{volume} of $\L$. In fact our results apply for a much more general class of energy functionals, which we call \emph{positive Deligne functionals}, with our slope formulas then involving the \emph{positive intersection product} of Boucksom--Favre--Jonsson when the subgeodesic has minimal singularities \cite{bfjteissier}. We actually prove slope formulas for very general singularity types, beyond the minimal singularity setting, employing recent work of Darvas--Xia and Xia who give an algebraic interpretation of mixed Monge--Amp\`ere masses \cite{darvas-xia-psef, xia-okounkov} (see also Trusiani \cite[Section 3]{AT}, who independently introduced the singularity classes relevant to these results). These sorts of slope formulae in the theory of Deligne pairings---which go back to work of Phong--Ross--Sturm in the ample setting \cite{PRS}---play a crucial role in the general Yau--Tian--Donaldson conjecture addressing the existence of constant scalar curvature K\"ahler metrics in ample classes \cite{BHJ2}.

The slope formula we prove for the remaining component $L(\phi)$ of the Ding functional generalises a key result  of Berman in the case $-K_X$ is ample \cite{bermankpoly}. Our generalised version crucially involves a measure of the \emph{singularities} of the metric on the total space (not appearing in the work of Berman).  The most important situation is again when the metric on the total space itself has minimal singularities, meaning that that the slope formula involves algebro-geometric invariants of the asymptotic linear system $\|\L\|$, though we again prove our slope formula for very general classes of singularities.

\subsection*{Beyond the K\"ahler--Einstein problem} While the notion of Ding stability is central in the theory of K\"ahler--Einstein metrics on Fano manifolds, the notion of K-stability also plays an important role, not least because it makes sense for a \emph{general} polarised variety $(X,L)$, in the sense that $L$ is an ample line bundle on $X$. K-stability is predicted to be, through the Yau--Tian--Donaldson conjecture,  the algebro-geometric counterpart to the existence of constant scalar curvature K\"ahler metrics in $c_1(L)$.

Many of the ideas and techniques we develop here similarly apply beyond the K\"ahler--Einstein setting, and in particular in Section \ref{sec:dingstability} (especially Remark \ref{rmk:kstab}) we mention that there is also a natural notion of \emph{uniform K-stability} for a pair $(X,L)$ with $L$ assumed to be \emph{big} rather than ample. While this notion plays no role in the present work, we expect that it will do so in the future, and it is enticing to ask whether there is a corresponding theory of constant scalar curvature K\"ahler metrics in the big setting. Furthermore, we expect that our theory of positive Deligne functionals and associated slope formulas to be relevant beyond the K\"ahler--Einstein setting in much the same way as the work of Boucksom--Hisamoto--Jonsson, which our work generalises \cite{BHJ2}.

\subsection*{Moduli} Ding stability and K-stability have achieved their prominence for two reasons: the first is their relationship with the existence of K\"ahler--Einstein metrics and constant scalar curvature metrics, while the second is their relationship with the construction of \emph{moduli spaces of projective varieties}. For the latter we refer to \cite{xu-survey} for a survey. It is thus natural to ask the relevance of  Ding stability and K-stability for big classes to moduli theory. We speculate slightly imprecisely, and slightly optimistically, in this direction.

We first recall  our claim that the notions of K-stability and Ding stability should be thought of as properties of equivalence classes of projective varieties along with a line bundle under the equivalence relation generated by \emph{birational equivalence}: here we say that $(X_1,L_1) \sim (X_2,L_2)$ if there exists a $(Y,L_Y)$ with birational morphisms $\pi_1: Y \to X_1$ and $\pi_2: Y \to X_2$ such that $$\pi_{1*}L_Y = L_1 \textrm{ and } \pi_{2*}L_Y = L_2$$ and the volume remains fixed, in the sense that $$\vol(L_1) = \vol(L_2) = \vol(L_Y).$$ Analytically, this equivalence relation is motivated by the fact that the space of finite energy metrics in $c_1(L_1)$ is naturally identified with that of $c_1(L_2)$, through each being identified with that of $c_1(L_Y)$ \cite{DNFT}. Algebraically, equivalence relation is very natural from the perspective of birational geometry, as we now explain.

Recall that when $K_X$ is big, $X$ is said to be of \emph{general type}. Supposing $\pi: Y \to X$ is a birational morphism, we can write $$K_Y = \pi^*K_X + \sum a_i E_i,$$ and the condition $\pi_*K_Y = K_X$ asks that $a_i\geq 0$; this holds provided $X$ has \emph{canonical singularities} (in particular it is automatic when $X$ is smooth). The volume condition then holds through the resulting  identification $$H^0(Y,kK_Y) \cong H^0(X,kK_X).$$ Thus we have $$(X,K_X) \sim (Y,K_Y)$$ provided $X$ and $Y$ are birational and have canonical singularities. As one of the main results of the minimal model programme, each such equivalence class contains a unique representative $(X_{\can},K_{X_{\can}})$ with $K_{X_{\can}}$ ample \cite{BCHM}; this is called the \emph{canonical model} of $(X,K_X)$, it has canonical singularities in general.  Perhaps the most important application of this programme is  the construction of a projective, separated moduli space of varieties with ample canonical class and fixed volume \cite{JK}. To be more precise one must carefully define the notion of a family and the class of singularities of $X$ allowed. But what we note here is that---assuming the right notion of a family---this produces a moduli space of equivalence classes of varieties $(X,L)$ containing a representative with $K_X$ big.

It is thus reasonable to ask whether there exists a moduli space of birational equivalence classes of K-stable varieties $(X,L)$ with $L$ big, which seems most tractable in the situation $L=-K_X$ (for some representative of the birational equivalence class). In particular:

\begin{conjecture}\label{intro-con}
There is a projective, separated moduli space of birational equivalence classes of uniformly K-stable (or uniformly Ding stable) varieties $(X,L)$ such that each equivalence class contains a member with $L=-K_X$ big.
\end{conjecture}

We should say immediately that the conjecture is speculative, and its main role is to highlight our broader claim that moduli problems in the setting of big line bundles rather than ample line bundles can (and should) be studied. The conjecture above is not even well-formed as it stands, as we have not defined what it means to have a family of such varieties (which likely requires at least that various associated multiplier ideal sheaves vary in a flat manner). Similarly projectivity---or more precisely \emph{properness}---is particularly subtle, not least as the anticanonical ring $$\bigoplus_{k \geq 0} H^0(X,-kK_X)$$ of a variety with $-K_X$ big is not finitely generated in general \cite{FS}. K-semistable Fano varieties often have better properties than arbitrary Fano varieties, and  we ask: does a  \emph{Ding semistable} (or \emph{K-semistable}) variety $(X,-K_X)$ with $-K_X$ big have finitely generated anticanonical ring? We hope that  Conjecture \ref{intro-con} and this question serve as a guides to the relevance of our work to algebraic geometry. 

We end by noting that there is a variant of this notion of birational equivalence for projective varieties endowed with big line bundles, which we explain in Remark \ref{KKL}, and which we call \emph{strong birational equivalence}. This seems equally natural from the perspective of birational geometry, and is closely motivated by work of Kaloghiros--Kuronya--Lazi\'c \cite{KKL}; it is, however, not precisely the correct notion to relate metrics on $L_1$ and $L_2$ in general.

\subsection*{Parallel and subsequent work} We end by briefly describing parallel, independent work of Darvas--Zhang and Trusiani, and subsequent work of Xu, on the questions addressed here; while in each of these the motivation is similar, the techniques are quite different. 

The work of Darvas--Zhang \cite{darvaszhang} is perhaps closest to ours in motivation, however rather than developing a theory of test configurations and associated Ding invariants, they develop a theory of \emph{delta invariants} in the big setting (involving divisorial valuations on $X$). They then show that if $-K_X$ is big with delta invariant strictly larger than $1$, then $X$ admits a K\"ahler--Einstein metric; thus their main result is in the \emph{opposite} direction to our Theorem \ref{intromainthm}. We note that they give an explicit example of a Kähler-Einstein manifold with big anticanonical bundle which admits a K\"ahler--Einstein metric (\cite{darvaszhang}, after Corollary 1.3), and also note that there is overlap in our analytic results proving Theorem \ref{introanalytichm} and their work. 

As already mentioned, Trusiani has developed a very complete analytic theory of K\"ahler--Einstein metrics with prescribed singularities on Fano manifolds \cite{AT}. In related work he also develops an algebro-geometric counterpart to this theory, by developing notions of K-stability and Ding stability in the prescribed singularity setting, and connects them to the existence of K\"ahler--Einstein metrics with prescribed singularities \cite{trusiani-ytd}. More precisely, using some of the ideas introduced in our work (namely the slope formulas) he shows that the existence of a unique K\"ahler--Einstein metric with prescribed singularities implies uniform Ding stability relative to the singularity type, and that this in turn implies a version of uniform K-stability relative to the singularity type he introduces. His (independent) notion of a test configuration in this setting is closely related to the one used here, although there are differences. In addition Trusiani relates these notions to delta invariants with prescribed singularities, and---using several new ideas along with some techniques of Darvas--Zhang---proves an existence result for K\"ahler--Einstein metrics with prescribed singularities under a delta invariant assumption.

We finally mention work of Xu \cite{xu}, which is instead completely algebro-geometric. His main result answers a question raised in the first version of this paper and asked above: he shows that the anti-canonical ring of a K-semistable variety with $-K_X$ big \emph{is} finitely generated. Xu further shows that taking Proj of the anticanonical ring of $X$ produces a genuine (possibly singular) Fano variety $Z$, that uniform Ding stability of $(Z,-K_Z)$ is equivalent to uniform Ding stability of $(X,-K_X)$ in our sense, and that our notion of uniform Ding stability is equivalent to asking that the delta invariant be greater than one, connecting the approach of this paper and that of Darvas--Zhang. Using these results, the second version of the work of Darvas--Zhang explains the analytic counterpart to this  \cite[Section 6]{darvaszhang}, connecting the existence of K\"ahler--Einstein metrics on $X$ and $Z$ (much as happens in the general type setting),  giving another approach to the theory of K\"ahler--Einstein metrics in the big setting;  we emphasise again that many of the techniques and ideas we develop and the results we prove should play important roles beyond the K\"ahler--Einstein setting.

\subsection*{Acknowledgements} We thank S\'ebastien Boucksom, Eveline Legendre, Calum Spicer, Mingchen Xia and Chenyang Xu for helpful discussions. In addition we are very grateful to T\'amas Darvas and Kewei Zhang for sharing their work \cite{darvaszhang} with us and to Antonio Trusiani for explaining his forthcoming work to us and for many useful comments. RD was funded by a Royal Society University Research Fellowship, while RR was funded by a postdoctoral fellowship associated to the aforementioned Royal Society University Research Fellowship.

\section{Positive Deligne pairings and finite-energy spaces}

There are three goals to the current section. Firstly, we recall some foundational aspects  of pluripotential theory and K\"ahler geometry in big cohomology classes (for which a standard reference is \cite{JPD}). Secondly, we explain how resulting analytic volumes---integrals of \emph{Monge--Amp\`ere masses}---compare to corresponding algebraic invariants. The only truly new material is the third goal, which extends the theory of \emph{Deligne metrics} and {Deligne pairings} to the setting of big classes. 

Throughout this section, we will denote by $(X,\omega)$ a projective K\"ahler manifold of dimension $n$, and we will \emph{not} assume any sort of positivity of its anticanonical class $-K_X$. We will use additive notation for tensor products of line bundles on $X$, i.e. if $L$, $L'$ are two such line bundles, then $kL-L'$ denotes the tensor product $L^{\otimes k}\otimes (L')^{-1}$.

\subsection{Preliminaries} 

We recall that a cohomology class $\alpha\in H^{1,1}(X,\bbr)$ is said to be \textit{pseudoeffective} (or \textit{psef}) if it contains a positive closed $(1,1)$-current, and \textit{big} if it contains a closed $(1,1)$-current $\theta\geq \varepsilon \omega$ for some $\varepsilon>0$. If $\alpha=c_1(L)$ for some $\bbq$-line bundle $L$ on $X$, then the bigness condition for $\alpha$ is equivalent to maximal growth of sections of powers of $L$, i.e. if $m$ is such that $mL$ is a line bundle, then $L$ is big if and only if its \textit{volume}
$$\vol(L):=\lim_{k\to\infty}\frac{n!\dim H^0(X,kmL)}{(km)^n}$$
is strictly positive.

\bigskip If $\alpha$ is a pseudoeffective class represented by a positive current $\theta$, any other closed positive $(1,1)$-current in $\alpha$ can be written as $\theta+dd^c\phi$ for some unique usc function $\phi:X\to\bbr\cup\{-\infty\}$. We define the set $\PSH(X,\theta)$ of \textit{$\theta$-plurisubharmonic functions} (or \textit{$\theta$-psh} functions) to be the set of usc functions $\phi$ such that $\theta+dd^c\phi$ is a closed positive $(1,1)$-current. If $[\theta]=c_1(L)$ for some $\bbq$-line bundle $L$, psh functions can be identified with potentials of positive metrics on $L$, and we will often use the same terminology for, interchangeably, $\theta$-psh functions, currents in the class of $\theta$, or metrics on $L$ (remaining careful that currents correspond to metrics modulo the rescaling action of $\mathbb{R}$).

\bsni The class $\PSH(X,\theta)$ has the property of being closed under finite maxima, convex combinations, and decreasing limits; furthermore, the usc regularisation of a family of $\theta$-psh functions is always $\theta$-psh. Thus one can always define the function
$$V_\theta:=\mathrm{usc}\sup\{0\geq \phi\in\PSH(X,\theta)\}\in\PSH(X,\theta),$$
which is in a sense the \textit{least singular} $\theta$-psh function. Any other $\theta$-psh function $\phi$ with $|\phi-V_\theta|<C$ for some $C\geq 0$ is said to have \textit{minimal singularities}. Equivalently, $\phi\in\PSH(X,\theta)$ has minimal singularities if and only if, for any $\psi\in\PSH(X,\theta)$, there exists $C\geq 0$ such that $\phi\geq\psi-C$. We write $\PSH_\mathrm{min}(X,\theta)$ the set of $\theta$-psh functions with minimal singularities, and $\PSH_{\mathrm{min},0}(X,\theta)$ the set of $\theta$-psh functions with minimal singularities and with supremum equal to $0$ (hence smaller than $V_\theta$). If $T=\theta+dd^c\phi$ for some $\theta$-psh function $\phi$ with minimal singularities, we will say that it is a \textit{current with minimal singularities}, and will often use the notation $T_{\mathrm{min}}$. (Similarly, in the line bundle case, we will speak of \textit{metrics with minimal singularities}.)

\bigskip There are two important generalisations of the intersection product of nef line bundles to more general psef line bundles. In \cite{begz}, generalising the fundamental work of Bedford-Taylor \cite{bedfordtaylor}, the authors define the \textit{non-pluripolar product} of $n$ closed positive $(1,1)$-currents $\theta_1,\ldots,\theta_n$
$$\langle \theta_1\wedge\ldots\wedge\theta_n\rangle$$
as a closed positive $(n,n)$-current on $X$, which puts no mass on pluripolar subsets of $X$ (i.e. singularity loci of psh functions on $X$). In particular, if for all $i$, $[\theta_i]=c_1(L_i)$ for some big $\bbq$-line bundle $L$ on $X$, then one can define their (analytic) positive intersection product
$$\langle L_1,\ldots, L_n\rangle_{\mathrm{an}}:=[\langle T_{1,\mathrm{min}}\wedge\ldots\wedge T_{n,\mathrm{min}}\rangle]\in H^{n,n}(X,\bbr),$$
where each $T_{i,\mathrm{min}}$ is a current with minimal singularities in $c_1(L_i)$. It is shown in \cite{begz} that this product is independent of the choice of such currents with minimal singularities, and that it recovers the usual intersection product $(L_1\cdot\dots\cdot L_n)$ if all the $L_i$ are nef.

\bsni On the other hand, one can define an \textit{algebraic} positive intersection product as in \cite{bfjteissier}, where $\langle L_1,\dots,L_n\rangle_{\mathrm{alg}}\in N^n(X)$ is identified with the least upper bound of the set of intersection numbers of the form
$$((\pi^*L_1-D_1)\cdot\ldots\cdot(\pi^*L_n-D_n))$$
where $\pi:X'\to X$ is some composition of blow-ups of $X$, and each $D_i$ is an effective $\bbq$-Cartier divisor on $X'$ such that $\pi^*L_i-D_i$ is nef. As explained in \cite[Remark 3.2]{xiaoconvex}, if $X$ is smooth projective, then
\begin{equation}
\langle L_1,\dots,L_n\rangle_{\mathrm{alg}}=\langle L_1,\dots,L_n\rangle_{\mathrm{an}},
\end{equation}
and we just write $\langle L_1,\dots,L_n\rangle$. Note in particular \cite{bouvolume} that we have
$$\vol(L)=\langle L^n\rangle.$$
If $\theta$ is a fixed closed positive $(1,1)$-current, and $\phi\in\PSH(X,\theta)$, we will write
$$\MA_\theta\langle \phi\rangle:=\langle (\theta+dd^c\phi)\wedge\ldots\wedge(\theta+dd^c\phi)\rangle,$$
which we call the \textit{Monge-Ampère operator}.

\bigskip Lastly, we state an important property of non-pluripolar products: they satisfy the following \textit{integration by parts formula} (\cite[Theorem 1.14]{begz}, see also \cite{xiaibp}, \cite{luibp}):
\begin{proposition}Let $\theta_0,\dots,\theta_n$ be closed positive $(1,1)$-currents on $X$ and $\phi_i,\psi_i$, $i=0,1$ be $\theta_i$-psh functions with $|\phi_i-\psi_i|<C$ for some $C\geq 0$. Then,
\begin{equation}
\int_X (\phi_1-\psi_1) dd^c(\phi_0-\psi_0)\wedge \theta_2\wedge\ldots\wedge \theta_n=\int_X (\phi_0-\psi_0) dd^c(\phi_1-\psi_1)\wedge \theta_2\wedge\ldots\wedge \theta_n,
\end{equation}
where $dd^c$ of bounded functions is understood in the sense of Bedford-Taylor.
\end{proposition}

\subsection{More general volumes}\label{sect_23}

We recall various results around algebro-geometric interpretations of Monge-Amp\`ere masses. These results  will be used to give a more truly algebraic interpretation of the mixed Monge--Amp\`ere masses which will appear in our slope formulas for Deligne functionals in the case the metrics have arbitrary singularities, but are not needed for our main results around the existence of K\"ahler--Einstein metrics, where the metrics will have minimimal singularities and the  slope will be the usual positive intersection products.

Consider a $\theta$-psh function $\psi$, for $\theta\in c_1(L)$ a smooth reference metric. We first assume that $\psi$ has \emph{analytic singularities}, in the sense that locally around any point there exist holomorphic functions $f_j$ such that the function $$\psi - \log \left(\sum^m_{j=1} |f_j|^2\right)$$ is locally bounded. Denote by $\J(c\psi)$ the \emph{multiplier ideal sheaf} of $c\psi$ for $c \in \R \geq 0$, namely $$\J(\psi)(U) = \{ f \in \scO_X(U): |f|^2 e^{-c\psi} \in L^1_{\mathrm{loc}}\},$$ which is a coherent ideal sheaf by Nadel. 

\begin{theorem}[Bonavero]\cite{bonavero} The Monge--Amp\`ere mass of $\theta_{\psi}=\theta+\ddc \psi$ is given by $$\int_{X} \MA_\theta\langle \psi\rangle   = \lim_{k\to \infty} \frac{n!\dim H^0(X,L^k \otimes \J(k\psi))}{k^n}.$$\end{theorem}

To view the multiplier ideal sheaves appearing in this statement more algebraically, for an ideal sheaf $\mfa$ one takes a resolution of $\pi: Y \to (X,\mfa)$ so that $\pi^{-1}(\mfa)= \scO_Y(-F)$ for $F$ a divisor on $Y$ and sets $$\J(c\mfa) = \pi_*\scO_Y(K_Y - \pi^*K_X - \lfloor cF\rfloor),$$ where $\lfloor cF\rfloor$ is the integral part (or round-down) of $cF$. Using this one defines the multiplier ideal sheaf $\J(c|mL|)$ as the multiplier ideal sheaf of ($c$ times) the base ideal of $|mL|$, and subsequently defines the \emph{asymptotic multiplier ideal sheaf} $\J(c\|L\|)$ to be the unique maximal element over all $m$ of the collection of ideal sheaves $\J(\frac{c}{m}|mL|)$, which exists by a ascending chain argument \cite[Lemma 11.1.2]{RL-book}. These  multiplier ideal sheaves can also be defined via valuations, and we will take this perspective in Section \ref{sec:dingstability}. The analytic and algebraic volumes are then related as follows.

\begin{theorem}\cite[Section 1.6.1]{guanzhou}\label{gz-openness} Let $\psi$ have minimal singularities. Then for all $c \geq 0$ we have an equality of multiplier ideal sheaves $$\J(c\psi) = \J(c\|L\|).$$


 \end{theorem}

We now consider the case of a general singular $\theta$-psh function ${\psi}$, no longer assumed to have analytic singularities, in which the theory is due to Trusiani \cite{AT} and Darvas--Xia \cite{darvas-xia-psef}. In particular, the line bundle $L$ is assumed to be pseudoeffective. Here the integral of the Monge--Amp\'ere measure of the metric $\psi$ itself is no longer the appropriate quantity to link with algebro-geometric volume invariants, and instead we must consider its $\J$\emph{-model envelope} $$P_{\J}(\psi) = \sup \{\psi' \in \PSH(\theta): \psi' \leq 0 \textrm{ and } \J(c\psi) = \J(c\psi') \textrm{ for all } c \geq 0\},$$ which is a $\theta$-psh function. 

\begin{definition} We say that $\psi$ is $\J$\emph{-model} if $P_{\J}(\psi) = \psi$. \end{definition}

By Darvas--Xia \cite[Theorem 1.4]{darvas-xia-closure} and Trusiani \cite{trusiani3}, the volume bound $$\int_X\MA_\theta  \langle \theta_{\psi} \rangle \leq \int_X \MA_\theta\langle \theta_{P_{\J}(\psi)}\rangle$$ holds, but inequality is possible \cite[Example 2.19]{darvas-xia-closure}.

\begin{theorem}[Darvas--Xia]\label{DX-volume}\cite[Theorem 1.1]{darvas-xia-psef}  The Monge--Amp\`ere mass of $\theta_{P_{\J}(\psi)}$ is given by $$\int_X  \MA_\theta\langle P_{\J}(\psi)\rangle = \lim_{k\to \infty} \frac{n!\dim H^0(X,kL \otimes \J(k\psi))}{k^n}.$$\end{theorem}

We finally turn to more general mixed Monge--Amp\`ere measures, where instead the interpretation we give is as an intersection number over the \emph{Riemann--Zariski space} \cite{bfjteissier}. Our discussion will follow work of Xia quite closely \cite[Theorem 5.2]{xia-pp} (see also \cite{xia-vb}), and so our discussion is brief. 

\newcommand{\mfX}{\mathfrak{X}}

We first recall the basic theory of the Riemann--Zariski space associated to a smooth projective variety $X$, which is the locally ringed space denoted $\underline X$ defined by $$\underline X = \varprojlim_Y Y,$$ where $Y \to X$ is a birational model of $X$. For further details we refer to \cite{bfjteissier, favre-dang, xia-pp}.  The main point is the definition of a \emph{Weil $b$-divisor}, which is an element of $$ \varprojlim_Y \Weil(Y),$$  with $ \Weil(Y)$ denoting the set of numerical classes of $\R$-divisors on $Y$. 

Let $\theta \in c_1(L)$  with $\phi \in \PSH(\theta)$. To $\phi$ we can associate a Weil $b$-divisor $\divisor_{\underline X} \phi$ in the following manner. For a fixed birational model $Y\to X$ define $$\divisor_Y \phi = \sum_E \nu_E(\phi) E,$$ where $\nu_E(\phi)$ is the generic Lelong number of $\pi^*\phi$ along $E$; this produces a well-defined numerical class on $Y$ by \cite[Remark 2.10]{xia-pp}. Thus one can define a Weil $b$-divisor by $$\divisor_{\underline X} \phi = \varprojlim_Y\divisor_Y \phi,$$ and hence a Weil $b$-divisor $L- \divisor_{\underline X} \phi$.

Consider now for $j=1,\hdots,n$ classes $\theta_j \in c_1(L_j)$ with $\phi_j\in \PSH(\theta_j)$. For each $j$ we associate a Weil $b$-divisor $L_j - \divisor_{\underline X} \phi_j$, and we define $$(L_1 - \divisor_{\underline X} \phi_1) \cdot \ldots \cdot (L_n - \divisor_{\underline X} \phi_n) =\varprojlim_Y (L_1 - \divisor_{Y} \phi_1) \cdot \ldots \cdot (L_n - \divisor_{Y} \phi_n),$$ which exists as the right hand side is decreasing in nets. 

The main result we will use is the following.

\begin{theorem}\cite[Theorem 10.13]{xia-vb} Suppose for each  for $j=1,\hdots,n$ the $\theta_j$-psh function $\phi_j$ is $\J$-model and $\theta_j+\ddc \phi_j$ has strictly positive Monge--Amp\`ere mass. Then $$\int_X \langle (\theta_1+\ddc \phi_1)\wedge \ldots \wedge (\theta_n+\ddc \phi_n)\rangle =  (L_1 - \divisor_{\underline X} \phi_1) \cdot \ldots \cdot (L_n - \divisor_{\underline X} \phi_n).$$ \end{theorem}

\begin{remark}The first version of the current work contained a proof of this result, as an extension of a result of Xia  \cite[Theorem 5.2]{xia-pp}  and Dang--Favre \cite[Proof of Theorem 6 (3)]{favre-dang} in the case $\theta_1+\ddc \phi_1 =\hdots = \theta_n+\ddc \phi_n$. Xia informed us that a proof of the general case was also contained in his simultaneous work \cite{xia-vb}, and since our proof was a variant of Xia's proof \cite{xia-pp}, we have chosen to omit the details in the current version. \end{remark}

\subsection{Positive Deligne pairings}

In the relative setting, namely with $X\to Y$ (say) a smooth flat morphism of projective varieties of relative dimension $n$, the \textit{Deligne pairing} construction generalises the intersection product of line bundles on $X$ to this relative setting, yielding an ``intersection line bundle''
$$\langle L_1,\dots,L_{n+k}\rangle_{X/Y}$$
over $Y$ for $n+k$ (relatively) ample line bundles over $X$. Historically, the case of relative dimension one was settled in Deligne's original article \cite{delignedetcoh}, and then extended to the general case by Elkik \cite{elkik1, elkik2}. It was first used in K\"ahler geometry by \cite{shouwu-zhang, PRS}. 

An important aspect of the construction is that it extends to the \textit{metrised} case: if each $L_i$ is endowed with a continuous metric $\phi_i$, one obtains a continuous metric \cite{moriwaki}
$$\langle \phi_1,\dots,\phi_{n+k}\rangle_{X/Y}$$
on the line bundle $\langle L_1,\dots,L_{n+k}\rangle_{X/Y}$. The construction is symmetric, multilinear, and satisfies the change of metric formula
$$\langle \phi_1,\dots,\phi_{n+k}\rangle_{X/Y} - \langle \phi_1',\dots,\phi_{n+k}\rangle_{X/Y} = \pi_*\left((\phi_1-\phi_1')(dd^c\phi_2\wedge\dots\wedge dd^c\phi_{n+k})\right).$$
Note that it has been extended to the case of general Kähler classes on non-projective Kähler manifolds in \cite[Definition 2.1]{zak}, \cite{zaknotthesis} (see also \cite[Section 3]{dervan-ross}) and to singular metrics on (relatively) ample line bundles in \cite{rebdiscs}.

\bsni In this section, we show how to extend this construction to the case of \textit{big} cohohomology classes and metrics with minimal singularities, in a way that is compatible with the positive intersection pairing. Although we will only consider globally big classes in the present article, we also state our definitions in the relatively big setting.

\begin{definition}
Let $\pi:X\to Y$ be a surjective flat holomorphic morphism of relative dimension $n$, whose fibres we denote by $X_y$, $y\in Y$. Let $\alpha$ be a cohomology class on $X$. We say that it is \textit{relatively big} if it is represented by a closed $(1,1)$-current $\theta$ such that its restriction $\theta_y:=\theta|_{X_y}$ represents a big cohomology class on $X_y$. We say that a current $T\in\alpha$ is \textit{$\theta$-relatively psh} if $T=\theta+dd^c\phi$ and $\phi_y:=\phi|_{X_y}$ is $\theta_y$-psh for each $y\in Y$. We say that it has \textit{relative minimal singularities} (in the class of $\theta$) if it is $\theta$-relatively psh and $\phi_y$ has minimal singularities in the class of $\theta_y$ for each $y\in Y$.
\end{definition}

\begin{definition}
Let $\pi:X\to Y$ be a surjective flat holomorphic morphism of relative dimension $n$ with smooth fibres. Let $\theta_0,\dots,\theta_{n+k-1}$ be closed $(1,1)$-currents on $X$ representing relatively big cohomology classes. For each $i=0,\dots,n+k-1$, let $\phi_i,\psi_i$ be $\theta_i$-relatively psh potentials with relative minimal singularities. We define the (difference of) \textit{positive Deligne pairing}(s) of the $(\theta_i,\phi_i,\psi_i)_i$ as
\begin{align*}
    &\llangle \phi_0,\dots,\phi_{n+k-1}\rrangle_{X/Y}-\llangle \psi_0,\dots,\psi_{n+k-1}\rrangle_{X/Y}:=\\
    &\pi_*\left((\phi_0-\psi_0)\,\langle(\theta_1+dd^c \phi_1)\wedge\dots\wedge(\theta_{n+k-1}+dd^c\phi_{n+k-1})\rangle\right)\\
    &+\pi_*\left( (\phi_1-\psi_1)\,\langle(\theta_0+dd^c \psi_0)\wedge(\theta_2+dd^c\phi_2)\wedge\dots\wedge(\theta_{n+k-1}+dd^c\phi_{n+k-1})\rangle\right)\\
    &+\dots\\
    &+\pi_*\left( (\phi_{n+k-1}-\psi_{n+k-1})\,\langle(\theta_0+dd^c \psi_0)\wedge\dots\wedge(\theta_{n+k-2}+dd^c \psi_{n+k-2})\rangle\right),
\end{align*} where $\pi_*$ denotes the fibre integral. If $Y$ is a point, we will often drop the subscript $X/Y$ and refer to these as \emph{positive Deligne functionals}, as those generalise the Deligne functionals of \cite{zak} to the big setting.
\end{definition}

Note also that we should make explicit the choice of reference currents $\theta_i$ in the notation; we choose not to in order to lighten notation.

\begin{remark}Note that, in light of the change of metric formula above, we recover the usual (difference of) Deligne pairings in the case where each $L_i$ is relatively nef. Since the sum of currents with minimal singularities does not in general have minimal singularities, the construction cannot be multilinear by nature.
\end{remark}

\begin{proposition}
Differences of positive Deligne pairings are symmetric, i.e. if $\sigma$ is a permutation of $\{0,\dots,n+k-1\}$, then
$$\llangle \phi_{\sigma(0)},\dots,\phi_{\sigma(n+k-1)}\rrangle-\llangle \psi_{\sigma(0)},\dots,\psi_{\sigma(n+k-1)}\rrangle=\llangle \phi_0,\dots,\phi_{n+k-1}\rrangle-\llangle \psi_0,\dots,\psi_{n+k-1}\rrangle.$$
\end{proposition}
\begin{proof}
We only treat the case where $Y$ is a point (and hence $X$ is smooth), as the general case follows from similar (but notationally heavier) computations. Since adjacent transpositions generate the symmetric group, it suffices to assume that $\sigma$ permutes some $j$ and $j+1$. For clarity we will write $\theta_\phi^i:=\theta_i+dd^c\phi_i,\,\theta_\psi^i:=\theta_i+dd^c\psi_i$. Then,
\begin{align*}
    &\llangle \phi_{0},\dots,\phi_j,\phi_{j+1},\dots,\phi_{n}\rrangle-\llangle \psi_{0},\dots,\psi_j,\psi_{j+1},\dots,\psi_{n}\rrangle\\
    &-(\llangle \phi_{0},\dots,\phi_{j+1},\phi_{j},\dots,\phi_{n}\rrangle-\llangle \psi_{0},\dots,\psi_{j+1},\psi_{j},\dots,\psi_{n}\rrangle)\\
    =&\sum_{i=0}^{j-1}\biggl(\int_X(\phi_i-\psi_i)\langle \theta_\psi^0\wedge\dots\wedge\theta_\psi^{i-1}\wedge\theta_\phi^{i+1}\wedge\dots\wedge\theta_\phi^j\wedge\theta_\phi^{j+1}\wedge\dots\wedge\theta_\phi^n\rangle\\
    &-\int_X(\phi_i-\psi_i)\langle \theta_\psi^0\wedge\dots\wedge\theta_\psi^{i-1}\wedge\theta_\phi^{i+1}\wedge\dots\wedge\theta_\phi^{j+1}\wedge\theta_\phi^{j}\wedge\dots\wedge\theta_\phi^n\rangle\biggl)\\
    &+\biggl(\int_X(\phi_j-\psi_j)\langle \theta_\psi^0\wedge\dots\wedge\theta_\psi^{j-1}\wedge\theta_\phi^{j+1}\wedge\dots\wedge\theta_\phi^n\rangle\\
    &-\int_X(\phi_{j+1}-\psi_{j+1})\langle \theta_\psi^0\wedge\dots\wedge\theta_\psi^{j-1}\wedge\theta_\phi^{j}\wedge\dots\wedge\theta_\phi^n\rangle\biggl)\\
    &+\biggl(\int_X(\phi_{j+1}-\psi_{j+1})\langle \theta_\psi^0\wedge\dots\wedge\theta_\psi^{j}\wedge\theta_\phi^{j+2}\wedge\dots\wedge\theta_\phi^n\rangle\\
    &-\int_X(\phi_{j}-\psi_{j})\langle \theta_\psi^0\wedge\dots\wedge\theta_\psi^{j+1}\wedge\theta_\phi^{j+2}\wedge\dots\wedge\theta_\phi^n\rangle\biggl)\\
    &+\sum_{i=k+1}^n\biggl(\int_X(\phi_i-\psi_i)\langle \theta_\psi^0\wedge\dots\wedge\theta_\psi^j\wedge\theta_\psi^{j+1}\wedge\dots\wedge\theta_\psi^{i-1}\wedge\theta_\phi^{i+1}\wedge\dots\wedge\theta_\phi^n\rangle\\
    &-\int_X(\phi_i-\psi_i)\langle \theta_\psi^0\wedge\dots\wedge\theta_\psi^{j+1}\wedge\theta_\psi^{j}\wedge\dots\wedge\theta_\psi^{i-1}\wedge\theta_\phi^{i+1}\wedge\dots\wedge\theta_\phi^n\rangle).
\end{align*}
By symmetry of non-pluripolar products, terms inside the sums in the first and last (fourth) groups cancel out. Summing the the second and third groups, and again using symmetry and multilinearity of non-pluripolar products, we find
\begin{align*}
    \int_X(\phi_j-\psi_j)\langle \theta_\psi^0\wedge\dots\wedge\theta_\psi^{j-1}\wedge(\theta_\phi^{j+1}-\theta_\psi^{j+1})\wedge\dots\wedge\theta_\phi^n\rangle\\
    -\int_X(\phi_{j+1}-\psi_{j+1})\langle \theta_\psi^0\wedge\dots\wedge\theta_\psi^{j-1}\wedge(\theta_\phi^{j}-\theta_\psi^j)\wedge\dots\wedge\theta_\phi^n\rangle,
\end{align*}
which vanishes by the non-pluripolar integration by parts formula.
\end{proof}

\begin{remark} Although out of the scope of the present article, we note that a slight modification of our definition can accommodate for the case of metrics with \emph{prescribed singularities} (originally introduced in \cite{rwn} and studied e.g. in \cite{trusiani2,trusiani3} in the ample case, and in \cite{xiamabuchi} in the big case).
\end{remark}

\subsection{Metric structure and geodesics}

We now fix a big cohomology class $\alpha$ and a smooth representative $\theta\in\alpha$.

\begin{definition}Given $\phi\in \PSH_{\mathrm{min}}(X,\theta)$, we define its \textit{Monge--Ampère energy}
$$E_\theta\langle \phi\rangle:=(n+1)\mi(\llangle \phi^{n+1}\rrangle - \llangle V_\theta^{n+1}\rrangle).$$
More generally, if $\phi\in\PSH(X,\theta)$, we can consider the sequence of approximants $\phi_k:=\max(\phi,V_\theta-k)$, which have minimal singularities and decrease pointwise to $\phi$. From the definition of $E_\theta$ it is clear that it is decreasing along pointwise decreasing sequences: we can then extend it by setting
$$E_\theta\langle \phi\rangle:=\lim_{k\to\infty} E_\theta\langle \phi_k\rangle.$$
Noting  that this can take the value $-\infty$, we define the \emph{finite energy space} $\cE^1(X,\theta)$ to be the set of $\theta$-psh functions with finite Monge-Ampère energy, and denote by $\cE^1_{0}(X,\theta)$ the set of psh functions of finite energy which have their supremum equal to $0$.
\end{definition}

 In particular, $V_\theta$ is the largest function in $\cE^1_{0}(X,\theta)$. By definition of $E$, one also has that $\PSH_{\mathrm{min}}(X,\theta)\subset \cE^1(X,\theta)$. Our construction via positive Deligne functionals agrees with that used in the pluripotential theory, e.g. as in \cite[Definition 2.7]{bbgz}. We collect the following important results related to $E$.
\begin{proposition}[{\cite[Propositions 4.3, 4.4]{bb}, \cite[Lemma 2.6]{bbgz}}]\label{propma}The Monge--Ampère energy satisfies the following properties:
\begin{enumerate}[(i)]
\item it is continuous along pointwise decreasing limits, Lebesgue-almost everywhere increasing limits, and uniform limits in $\PSH_{\mathrm{min}}(X,\theta)$;
\item it is concave with respect to the natural convex structure on $\PSH_{\mathrm{min}}(X,\theta)$ given by convex combinations;
\item its upper level sets in $\cE^1_{0}(X,\theta)$ are convex and compact with respect to the psh (i.e. $L^1$) topology.
\end{enumerate}
\end{proposition}

Following \cite{ddnlmetric}, one can endow the space $\cE^1(X,\theta)$ with a metric structure, using the following envelope construction.
\begin{definition}Let $\phi_0,\phi_1\in\PSH(X,\theta)$. We define
$$P_\theta(\phi,\psi):=\mathrm{usc}\sup\{\psi\in\PSH(X,\theta),\,\psi\leq\min(\phi_0,\phi_1)\}.$$\end{definition}
Such envelopes are usually called \textit{rooftop envelopes} in the literature (e.g. \cite{darmabuchi}). It has been shown in \cite[Theorem 2.10]{ddnlfullmass} that if both functions have finite energy, then their rooftop envelope has finite energy; furthermore, one easily sees that if both functions have minimal singularities, then so does their rooftop envelope.

\begin{definition}Given $\phi_0,\phi_1\in\cE^1(X,\theta)$, we define their $d_1$\emph{-distance} as
$$d_{1,\theta}(\phi_0,\phi_1):=E_\theta\langle \phi_0\rangle + E_\theta\langle \phi_1\rangle -2E_\theta\langle P_\theta(\phi_0,\phi_1)\rangle\geq 0.$$

\end{definition}

Note that, if $\phi_0\geq \phi_1$, then
$$d_{1,\theta}(\phi_0,\phi_1)=E_\theta\langle \phi_1\rangle - E_\theta\langle \phi_0\rangle.$$

This produces a metric by Darvas--di Nezza--Lu.

\begin{theorem}\cite[Theorems 3.6, 3.10]{ddnlmetric} The space $(\cE^1(X,\theta),d_{1,\theta})$ is a complete metric space.
\end{theorem}

It will also be important to relate this metric to geodesics in finite energy spaces.

\begin{theorem}[{\cite[Section 3.1, Theorem 3.12]{ddnlfullmass}, \cite[Proposition 3.13]{ddnlmetric}}]\label{thm_geodminsing}The space $(\cE^1(X,\theta),d_{1,\theta})$ is a geodesic metric space, i.e. given $\phi_0,\phi_1\in \cE^1(X,\theta)$, there exists a distinguished segment $t\mapsto \phi_t\in \cE^1(X,\theta)$ such that, for all $t,s\in[0,1]$,
$$d_{1,\theta}(\phi_t,\phi_s)=|t-s|d_{1,\theta}(\phi_0,\phi_1),$$
which furthermore satisfies the following properties:
\begin{enumerate}[(i)]
	\item if the endpoints belong to $\PSH_{\mathrm{min}}(X,\theta)$, then the geodesic remains in $\PSH_{\mathrm{min}}(X,\theta)$;
	\item in this case, the following estimate holds: there exists $C>0$ depending only on the endpoints, such that for all $t$, $s\in [0,1]$,
	$$\sup_X |\phi_t-\phi_s| \leq C|t-s|;$$
	\item the function $t\mapsto E_\theta\langle \phi_t\rangle$ is affine on $[0,1]$;
	\item seen as a function $\Phi$ on $X\times \cA$, with $\cA\subset \bbd$ an annulus such that $z\in \cA$ iff $-\log|z|\in [0,1]$, $\Phi$ is $\pi_X^*\theta$-psh;
	\item using the same notation, $\Phi$ satisfies the Monge-Ampère equation $$\langle(\pi_X^*\theta+dd^c\Phi)^{n+1}\rangle=0;$$
\end{enumerate}
\end{theorem}

\section{Variational approach to Kähler-Einstein currents}

\newcommand\phike{\phi_{\mathrm{KE}}}

Throughout this section, we assume that $-K_X$ is big, and we fix $\theta$ a smooth representative of $c_1(-K_X)$.

\subsection{Canonical measures}\label{sect_canmeasures}

We explain how psh metrics on $-K_X$ relate to measures on $X$ itself, analogously to the classical situation in which $-K_X$ is ample. Our discussion is close to that of (for example) Berman in the case $-K_X$ is ample \cite[Section 2.1.3]{bermankpoly},  as he allows $X$ to be singular. That is, Berman constructs a measure on the smooth locus of $X$ and extends by zero to all of $X$, and we will construct the appropriate measure similarly.

Let $\phi$ be a psh metric on $-K_X$ with minimal singularities, and consider a local chart $U \subset X$ lying in the complement of the polar locus $\{x \in X: \phi(x)=-\infty\}$ of $\phi$; the polar locus is contained in the asymptotic base locus of $-K_X$, which is algebraic, so such a choice of $U$ is possible away from a Zariski closed subset of $X$. Denote $z_1,\ldots,z_n$  the associated local coordinates, defining a local trivialising section $$\frac{\partial}{\partial z_1}\wedge\ldots\wedge \frac{\partial}{\partial z_n}$$ and hence producing a local potential $\phi_U$ for the psh metric $\phi$. We then obtain a measure $$\mu_{\phi} = e^{-\phi_U} i^{n^2} dz_1\wedge\ldots dz_n \wedge  d\bar z_1\wedge\ldots d\bar z_n,$$ which one checks is independent of coordinates and which is defined away from the polar locus of $\phi$. We extend this measure by zero to all of $X$, and often denote the resulting measure by $e^{-\phi}$. 

 Recall that the multiplier ideal sheaf $\J(\phi) $ of $\phi$ is independent of choice of metric with minimal singularities, and as in Theorem \ref{gz-openness} we have $$\J(\phi) = \J(\|-K_X\|).$$

\begin{definition}\label{def:klt}
We say that $\phi$ (or $\theta_{\phi})$ is \emph{klt} if and only if  $$\J(\phi) = \scO_X,$$ i.e. if the associated multiplier ideal sheaf is trivial. We sometimes say that $\|-K_X\|$ is \emph{klt} under the same condition. 
\end{definition}

Note then that the klt condition is completely algebro-geometric. The main consequence of this condition is that the associated measure $e^{-\phi}$ to a klt metric $\phi$ with minimal singularities has finite volume, and it is necessary for the variety to support Kähler--Einstein metrics.

Let $L$ be a line bundle on $X$. A similar construction applies for the adjoint bundle $L+K_X$, except one needs the additional input of a section $s \in H^0(X,L+K_X)$. Thus let $s$ be such a section and let $\phi$ be a psh metric on $L$ whose polar locus is contained in a Zariski closed set. Then $|s|^2e^{-\phi}$ can be interpreted as a measure on the complement of the polar locus of $\phi$ in much the same way, and one extends by zero to all of $X$ as before (compare again Berman \cite[Section 2.1.3]{bermankpoly} in the case $-K_X$ is ample but $X$ is singular). Provided $\I(\phi) = \scO_X$, the resulting measure again has finite volume. More generally, when $L$ is $\Q$-Cartier, with $s$ a section of $r(L+K_X)$ the same construction applies by considering the measure $|s|^{2/r}e^{-\phi}$.

\subsection{Weak Kähler--Einstein currents and the Ding functional}

The key analytic objects of study in the present work are \emph{K\"ahler--Einstein metrics}. We fix $\theta \in c_1(X)$ as before.

\begin{definition}We say that $\theta + dd^c\phi$ is a \emph{K\"ahler--Einstein metric} (or sometimes more accurately a \textit{weak Kähler--Einstein current}) if $\phi\in\cE^1(X,\theta)$, and if
\begin{equation}\label{eq_kem}
\MA_\theta\langle \phi\rangle=e^{-\phi+c},
\end{equation}
where $c$ is a real (normalising) constant. 
\end{definition}

Note that by \cite{begz}, such a current must have minimal singularities. Throughout this subsection, we further assume that $-K_X$ is klt in the sense of Definition \ref{def:klt}.

\begin{definition}We define the $L$-functional on $\PSH_{\mathrm{min}}(X,\theta)$ as
$$L_\theta\langle \phi\rangle:=-\log\int_X e^{-\phi},$$
which is finite-valued by virtue of the klt condition. It also extends to $\cE^1(X,\theta)$ by \cite[Theorem 1.1]{ddnlfullmass}. We then define the \textit{Ding functional} as
$$D_\theta\langle \phi\rangle:=L_\theta\langle\phi\rangle-\vol(-K_X)\mi E_\theta\langle\phi\rangle.$$
\end{definition}
As we will see in this section, this functional detects weak Kähler--Einstein currents. We first collect a few results.

\begin{proposition}\label{prop_continuityding}The Ding functional is lsc with respect to the psh topology, and $d_1$-continuous.
\end{proposition}
\begin{proof}
By the main theorem of \cite{demaillykollar} (as explained in \cite[Theorem 4.10]{AT}), we find that $L$ is weakly continuous and $d_1$-continuous. On the other hand, $-E$ is weakly lsc by \cite[Proposition 2.10]{begz} and $d_1$-continuous by definition.
\end{proof}

We also note the following immediate consequence of the main theorem of \cite{bpaun}.
\begin{theorem}\label{thm_dingconvex}The Ding functional is convex along weak geodesics in $\cE^1(X,\theta)$.
\end{theorem}

\begin{remark}
In the case $-K_X$ is ample, \emph{strict} convexity in the absence of holomorphic vector fields is proven by Berndtsson \cite{bob}, but is an open problem in the big setting. 
\end{remark}

The following result has been proven independently by Darvas-Zhang in \cite[Proposition 5.4]{darvaszhang}.

\begin{theorem}\label{thm_dingmin}Assume there exists a weak Kähler--Einstein current in $c_1(-K_X)$. Then, given $\phi\in\PSH_{\mathrm{min}}(X,\theta)$, $\theta+dd^c\phi$ is weak Kähler--Einstein if and only if $\phi$ minimises the Ding functional on $\cE^1(X,\theta)$.
\end{theorem}
\begin{proof}
We follow the corresponding argument in the case $-K_X$ is ample \cite[Theorem 6.6]{bbgz}. Assume that $\phi=\phi_{\mathrm{KE}}$ is the weak Kähler--Einstein metric in $c_1(-K_X)$. We want to show that $D_\theta\langle \phike\rangle \leq D_\theta\langle \psi\rangle$ for all $\psi\in\cE^1(X,\theta)$; by approximation and continuity of $D$ along decreasing sequences we may well assume that $\psi\in\PSH_{\mathrm{min}}(\theta)$. Since $D_\theta\langle \psi+c\rangle=D_\theta\langle \psi\rangle$ for all $\psi$ and $c$, we may also assume that $c=0$ in the Monge--Ampère equation \eqref{eq_kem}. Let $t\mapsto \phi_t$ be the weak geodesic with $\phi_0=\phike$ and $\phi_1=\psi$. Since $D$ is convex along geodesics by Theorem  \ref{thm_dingconvex}, it suffices to show that the derivative of $t\mapsto D_\theta\langle \phi_t\rangle$ at zero is nonnegative. Since both $\phike$ and $\psi$ have minimal singularities, the sequence of functions
$$u_t=t\mi(\phi_t-\phike)\leq |\phike-\psi|$$
decreases as $t\to 0$ to some function $u$ bounded above on $X$ by $|\phike-\psi|$. By concavity of $E$ and monotone convergence, we find
$$\frac{d}{dt}_{|t=0}E_\theta\langle \phi_t\rangle\leq \int_X u\,\MA\langle \phike\rangle=\int_X u\,e^{-\phike}.$$
One also has that
$$t\mi\left(\int_X e^{-\phi_t}-\int_X e^{-\phike}\right)=-\int_X u_t \frac{1-e^{\phike-\phi_t}}{\phi_t-\phike}e^{-\phike},$$
but $f(x):=x\mi (1-e^{-x})$ is continuous and $\phike-\phi_t$ is uniformly bounded on $X\times[0,1]$ by Theorem \ref{thm_geodminsing}, so that by dominated convergence we find on taking $t\to 0$ that
$$\frac{d}{dt}_{|t=0}\int_X e^{-\phi_t}=-\int_X u\,e^{-\phike}.$$
Since $D=L-E$, this together with the chain rule proves that the derivative of $D$ at zero is negative, hence the inequality $D_\theta\langle \phike\rangle \leq D_\theta\langle \psi\rangle$. The converse is proven in \cite[Theorem 4.22]{ddnlmonotonicity} (note that $V_\theta$ has small unbounded locus).
\end{proof}

\subsection{Existence of a weak Kähler--Einstein current implies analytic coercivity}

To prove the main result of this section, we will need a different functional, the \textit{$J$-energy}:
\begin{definition}For $\phi\in\PSH_{\mathrm{min}}(X,\theta)$, we define the \emph{$J$-energy}
\begin{align*}
J_\theta\langle \phi\rangle:&=\llangle \phi, V_\theta^n\rrangle-\llangle V_\theta^{n+1}\rrangle - E_\theta\langle\phi\rangle\\
&=\int_X (\phi-V_\theta)\MA_\theta\langle V_\theta\rangle-E_\theta\langle \phi\rangle.
\end{align*}

\end{definition}

Again, one sees the term $(\phi-V_\theta)\MA_\theta\langle V_\theta\rangle$ in  this functional is decreasing along pointwise decreasing sequences, while $E$ is already defined on $\cE^1$. Thus the $J$-energy extends to $\cE^1(X,\theta)$ as well, as $\MA_\theta(V_\theta)=1_{V_\theta\equiv 0}\theta^n$, so that $\theta$-psh functions are $\MA_\theta(V_\theta)$-integrable.

For reference we also introduce the $I$ and $I-J$-energies, which also admit natural interpretations as Deligne functionals.

\begin{definition} We define the $I$\emph{-energy} to be $$I_{\theta}\langle\phi\rangle = \int_X (\phi - V_{\theta})(\MA_{\theta}\langle V_{\theta}\rangle-\MA_{\theta}\langle\phi\rangle),$$ and set $(I_\theta-J_\theta)\langle\phi\rangle = I_\theta\langle\phi\rangle - J_\theta\langle\phi\rangle.$
\end{definition}

These energies are equivalent in the following sense:

\begin{lemma}[{\cite[Section 2.1]{bbgz}}]\label{IJ-uniform}
There are uniform bounds \begin{align*}J_\theta\langle\phi\rangle &\leq I_\theta\langle\phi\rangle \leq (n+1)J_\theta\langle\phi\rangle, \\ n^{-1}J_\theta\langle\phi\rangle &\leq I_\theta\langle\phi\rangle - J_\theta\langle\phi\rangle \leq nJ_\theta\langle\phi\rangle. \end{align*}
\end{lemma}

These functionals allow us define \emph{coercivity} of energy functionals.

\begin{definition}We say that $D$ is \textit{$d_1$-coercive} if there exists constants $\delta>0$, $C$ such that for all $\phi\in\cE^1_{0}(X,\theta)$, we have
$$D_\theta\langle \phi\rangle \geq \delta\cdot d_{1,\theta}(\phi,V_\theta)-C;$$
we say that it is \textit{$J$-coercive} if there exist $\delta>0$ and $C$ such that for all $\phi\in\cE^1_{0}(X,\theta)$, we have
$$D_\theta\langle \phi\rangle \geq \delta\cdot J_\theta\langle\phi\rangle-C.$$
\end{definition}

\begin{proposition}\label{prop_jd1comparison}$D$ is $J$-coercive if and only if it is  $d_1$-coercive.
\end{proposition}
\begin{proof}
It suffices to extend to the big case the estimate from \cite[Proposition 5.5]{darvasrubinstein}, i.e. to show that there exists $C>0$ with
$$d_{1,\theta}(\phi,V_\theta)-C\leq J_\theta\langle \phi\rangle \leq d_{1,\theta}(\phi,V_\theta)+C$$
for all $\phi\in\cE^1_{0}(X,\theta)$. We proceed as in \cite[Lemma 3.7]{AT}. By definition of $V_\theta$, we have $V_\theta\geq \phi$ on $\cE^1_{0}(X,\theta)$, hence
$$d_{1,\theta}(\phi,V_\theta)=-E_\theta\langle \phi\rangle\geq J_\theta\langle \phi\rangle.$$
For the reverse estimate, we use \cite[Corollary 3.4 $(i)$]{dntcontact} (setting $f=0$ in their notation), which gives $$\MA_\theta\langle V_\theta\rangle=(\theta
^n)|_{\{V_\theta=0\}},$$

hence
$$0\leq \int_X (\phi-V_\theta)\MA\langle V_\theta\rangle \leq \int_X |-\phi|\,(\theta^n).$$ By weak compactness of $\cE^1_0(X,\theta)$, the right-hand side is uniformly bounded by some constant $C$, hence using again the fact that $d_{1,\theta}(\phi,V_\theta)=-E_\theta\langle \phi\rangle$, we find that $-C+d_{1,\theta}(\phi,V_\theta)\leq J_\theta\langle\phi\rangle$, thus proving the left-hand estimate. 
\end{proof}

Of course, from Lemma \ref{IJ-uniform} it also follows that $d_1$-coercivity is equivalent to $I$ or $I-J$-coercivity. We now prove the main theorem of this section.

\begin{theorem}\label{thm_dingcoercive}Assume that $-K_X$ is big. If there exists a unique weak Kähler--Einstein current $\theta+dd^c\phi_{\mathrm{KE}}$ in $c_1(-K_X)$, then the Ding functional $D$ is $J$-coercive on $\cE^1_0(X,\theta)$.
\end{theorem}
\begin{proof}As in the proof \cite[Theorem D]{AT}, we consider
$$a:=\inf\{(D_\theta\langle \phi\rangle-D_\theta\langle\phi_{\mathrm{KE}}\rangle)/d_{1,\theta}(\phi,\phi_{\mathrm{KE}}),\,\phi\in\PSH_{\mathrm{min},0}(X,\theta),\,d_{1,\theta}(\phi,\phi_{\mathrm{KE}})\geq 1\},$$
and we first show that $a>0$. 

Assuming $a=0$, we can find a sequence of metrics $\phi_k\in\cE^1_{0}(X,\theta)$ such that $$(D_\theta\langle \phi_k\rangle-D_\theta\langle\phi_{\mathrm{KE}}\rangle)/d_{1,\theta}(\phi_k,\phi_{\mathrm{KE}})\to 0$$ and $$d_{1,\theta}(\phi_k,\phi_{\mathrm{KE}})\geq 1.$$ If we let $t\mapsto \phi_{t,k}$ be the unit speed geodesic parametrised by $[0,d_{1,\theta}(\phi_k,\phi_{\mathrm{KE}})]$ with $\phi_{0,k}=\phi_{\mathrm{KE}}$ and $\phi_{d_{1,\theta}(\phi_k,\phi_{\mathrm{KE}}),k}=\phi_k$, and $\psi_k:=\phi_{1,k}$, then by convexity of the Ding functional we find that $D_\theta\langle \psi_k\rangle \to_{k\to\infty} D_\theta\langle \phi_{\mathrm{KE}} \rangle$. 

Because by definition $d_{1,\theta}(\psi_k,\phi_{\mathrm{KE}})=1$, we also have that the family $(\psi_k)_k$ is contained in a (sup-normalised) upper level set of the Monge-Ampère energy. By weak compactness of such level sets (Proposition \ref{propma}), we can assume up to passing to a subsequence that $\psi_k\to_{k\to\infty}\psi\in\cE^1_{0}(X,\theta)$ in the psh topology. By lower semicontinuity of $D$, this together with the above convergence implies $D_\theta\langle\psi\rangle\leq D_\theta\langle \phi_{\mathrm{KE}}\rangle$. Since $\phi_{\mathrm{KE}}$ is Kähler--Einstein, it minimises $D$ by Theorem \ref{thm_dingmin}, hence $\psi=\phi_{\mathrm{KE}}$. Since the $L$-part of the Ding energy is continuous with respect to the psh topology, we have in particular that $E_\theta\langle \psi_k\rangle\to E_\theta\langle \phi_{\mathrm{KE}}\rangle$, hence $d_{1,\theta}(\psi_k,\phi_{\mathrm{KE}})\to 0$ (since the $d_1$-topology is equivalent to the psh topology together with convergence in energy) which contradicts the assumption that $d_{1,\theta}(\psi_k,\phi_{\mathrm{KE}})\geq 1$, thereby ensuring that $a>0$.

 Assuming now $a>0$ one sees that, if $\phi\in\cE^1_{0}(X,\theta)$ satisfies $d_{1,\theta}(\phi,\phi_{\mathrm{KE}})\leq 1$ then
\begin{align*}
&a\cdot d_{1,\theta}(\phi,V_\theta)-a\cdot\sup\{d_{1,\theta}(\phi',V_\theta),\,d_{1,\theta}(\phi',\phi_{\mathrm{KE}})\leq 1\}+D_\theta\langle \phi_{\mathrm{KE}}\rangle\leq D_\theta\langle\phi_{\mathrm{KE}}\rangle\leq D_\theta\langle \phi\rangle
\end{align*}
(where the second inequality follows again from the fact that Kähler--Einstein metrics minimise $D$). On the other hand, by definition of $a$, if $\phi\in\PSH_{\mathrm{min},0}(X,\theta)$ satisfies $d_{1,\theta}(\phi,\phi_{\mathrm{KE}})\geq 1$ then
$$a\cdot (d_{1,\theta}(\phi,V_\theta)-d_{1,\theta}(\phike,V_\theta))+D_\theta\langle \phi_{\mathrm{KE}}\rangle \leq D_\theta\langle \phi\rangle,$$
and this extends to $\cE^1_0(X,\theta)$ by approximation as usual; hence there exists $C$ such that
$$a\cdot d_{1,\theta}(\phi,V_\theta)-C\leq D_\theta\langle \phi\rangle$$
on $\cE^1_{0}(X,\theta)$. \end{proof}

\begin{remark}
We expect that K\"ahler--Einstein metrics are unique precisely when $\Aut(X)$ is finite. To justify this expectation, note firstly that if $\theta_{\phi}$ has minimal singularities in $c_1(X)$, then for $g \in \Aut(X)$, the pullback $g^*\theta_{\phi}$ has minimal singularities by \cite[Remark 3.7]{rwn}. Further, if $\theta_{\phi}$ is K\"ahler--Einstein, by then working away from the polar locus, it is clear that $g^*\theta_{\phi}$ is also  K\"ahler--Einstein. Thus uniqueness of  K\"ahler--Einstein metrics implies $\Aut(X)$ is finite.
\end{remark}

\section{Uniform Ding stability}\label{sec:dingstability}

\subsection{Ding stability}The goal of the present section is to introduce a notion of \emph{uniform Ding stability} for a normal projective variety $X$ with big anticanonical class $-K_X$. The notion of uniform Ding stability involves numerical invariants associated to \emph{test configurations}; since the notion of a test configuration makes sense and is of interest more generally, we begin by considering a normal projective variety $X$ along with a big $\bbq$-line bundle $L$ on $X$. 

Rather than a fixed $(X,L)$, we consider an equivalence class of such varieties in the following sense.

\begin{definition}\label{def:birational-equivalence} Let $(X,L)$ and $(Y,L_Y)$ be normal projective varieties endowed with big $\Q$-line bundles. We say that $(X,L)$ and $(Y,L_Y)$ are \emph{birationally equivalent} if there is a normal projective variety $Z$ along with a big $\Q$-line bundle $L_Z$ together with birational morphisms $p_X:Z \to X$ and $p_Y: Z \to Y$ such that $$p_{X*} L_Z = L \textrm{ and } p_{Y*} L_Z = L_Y,$$ and such that the following volume condition holds: $$\vol (L_Z) = \vol (L) = \vol (L_Y).$$ 
\end{definition}

Pictorially this is represented as:
\[
\xymatrix{ 
\quad & \ar@{->}[dl]_{p_X} (Z,L_Z)  \ar@{->}[dr]^{p_Y} & \quad \\
(X,L) & \quad  & (Y,L_Y).
}\]

All concepts we introduce will be birational invariants, so we will typically replace $(Y,L_Y)$ with $(Z,L_Z)$ and assume that $Y$ itself admits a morphism $p: (Y,L_Y) \to (X,L)$ defining the birational equivalence. A test configuration for $(X,L)$ is then essentially a $\C^*$-degeneration of an element of the birational equivalence class of $(X,L)$.

\begin{remark}\label{KKL} A closely related equivalence relation, which we call \emph{strong birational equivalence}, is given as follows. In the notation of Definition \ref{def:birational-equivalence}, we firstly ask that $q: Y\dashrightarrow X$ be a contraction with $q_*L_Y = L_X$. Writing $$p_Y^*L_Y = p_X^*L_X + E,$$ where $E$ is $p_X$-exceptional, we secondly ask that $E$ is effective. In the terminology of Kaloghiros--Kuronya--Lazi\'c \cite[Definition 2.3]{KKL} (whose work motivated the definition of strong birational equivalence), this asks that $q: Y \to X$ is $L_Y$-\emph{nonpositive}, and implies that for all $k \geq 0$ $$H^0(Y,kL_Y) \cong H^0(X,kL)$$ and hence that $\vol(L_Y) = \vol(L)$. This equivalence relation is perhaps more natural from the perspective of birational geometry, whereas our chosen equivalence relation (which can be seen as a sort-of asymptotic version of strong birational equivalence) is the right one to relate the K\"ahler geometry of $(Y,L_Y)$ and $(X,L)$.
\end{remark}

\begin{definition}
Suppose $p: (Y,L_Y) \to (X,L)$ defines a birational equivalence. A  \emph{test configuration} $(\X,\L)$ for $(X,L)$ (or more precisely for $p: (Y,L_Y) \to (X,L)$) is a normal projective variety $\X$, a $\Q$-line bundle $\L$ on $\X$ along with a surjective morphism $\pi: \X \to \pr^1$ such that
\begin{enumerate}[(i)]
\item there is a $\C^*$-action on $(\X,\L)$ making $\pi$ a $\C^*$-equivariant morphism;
\item there is a positive integer $k$ such that $k\cL$ is a line bundle, and a $\C^*$-equivariant isomorphism $(\X \backslash \{0\}, k\L|_{\X \backslash \{0\}}) \cong (Y,kL_Y)\times \C$, with the latter given the trivial $\C^*$-action on the first factor.\end{enumerate} \end{definition}

While the definition of a test configuration does not requite any positivity, in practice we will primarily be interested in \emph{big} test configurations. For this, passing to an equivariant  resolution of indeterminacy of the birational map $Y \times \pr^1 \dashrightarrow \X$ if necessary, we may assume $\X$ admits a morphism to $Y\times \pr^1$.

\begin{definition}
We say that $(\X,\L)$ is \emph{big} if  $\L$ is big and $\L-L_Y$ is effective.

\end{definition} 
\begin{remark}
The final condition, which asks that $\L$ minus the pullback of $L_Y$ to $\X$ is effective, is easily seen to be independent of choice of resolution of indeterminacy. The condition itself is harmless, as $\L - L_Y$ is supported in the central fibre, and so one can always modify $\L$ by adding a multiple of $\scO_{\pr^1}(1)$ to ensure effectiveness of $\L-L_Y$ while preserving the condition that $\L$ is big.
\end{remark}

We now specialise to the setting $L=-K_X$ is big. To define the Ding invariant, we will require a rather general version of the log canonical threshold, for which it is most convenient to take a valuative perspective (as for example in \cite[Section 3.3.2]{blum} or \cite{BFFU}). Considering a normal projective variety $X$ with a big line bundle $L$ and $\Q$-divisors $D_1, D_2$ (not necessarily effective), we decompose $D_1+D_2=P+N$ into an effective part $P$ and an antieffective part $N$. We then define \begin{align}\label{valuative-MIS}\J((X,D_1,\|L\|); &cD_2)(U) = \{f\in \scO_{X}(U):  \textrm{ there exists an  }\epsilon>0 \textrm{ such that for all }  E \\ &\ord_E(f) \geq (1+\epsilon) (\ord_E\|L\| + c\ord_E P)  - A_{(X,N)}(E)\}, \end{align}  where $E$ runs over prime divisors on birational models $Y \to X$ of $X$. The notation is as follows: firstly the \emph{log discrepancy} $A_{(X,N)}(E)$, is defined via a log resolution $\pi: Y \to X$ of $N$ and $E$ as $$A_{(X,N)}(E) =1+ \ord_E (K_{Y} - \pi^*(K_{X}+N)),$$ which  is defined provided  $K_{X}+N$ is $\Q$-Cartier (which in particular is necessary to make sense of the definition of the multiplier ideal sheaf). Secondly we have set $$\ord_E\|L\| = \lim_{k \to \infty} \frac{1}{k}\ord_E H^0(X,kL).$$
It is then clear from \cite[Corollary 10.15]{boucksom-l2} that triviality of $\J((X,D_1,\|L\|); cD_2)$ corresponds to the adequate analytic klt condition. 

Consider a test configuration $\pi:(\X,\L)\to(X,L)$, defined for a birational equivalence $p:(Y,L_Y)\to (X,L)$. We begin with the following observation (a particular case of which was proven in \cite[Section 3.1]{bermankpoly}). Suppose $r>0$ is such that $r\L$ is Cartier.

\begin{lemma}\label{lem_pushfwd} The pushforward $\pi_*(r(\L+K_{\cX/\pr^1}))$ is a line bundle. \end{lemma}

\begin{proof} Since $\pi$ is surjective, the pushforward is torsion-free, and since the base $\pr^1$ is one-dimensional, the direct image  $\pi_*(r(\L+K_{\X/\pr^1}))$ is actually locally free, hence corresponds to a vector bundle. But the general fibre is given as $$(\pi_*(r(\L+K_{\X/\pr^1})))_t\cong H^0(\X_t, \L_{t} +K_{\X_t})\cong H^0(Y,L_Y+K_Y).$$
Since $p:Y\to X$ is a birational equivalence, then $L_Y=\pi^*(-K_X)+E$ with $E$ some $\bbq$-divisor such that $\pi_*E=\cO_X$. Thus
$$H^0(Y,L_Y+K_Y)\cong H^0(Y,\pi^*(-K_X)+E+K_Y)\cong H^0(X,\cO_X)$$
by the projection formula. Since the right-hand vector space has dimension one, the result is proven.
 \end{proof}
By picking a local trivialising section for $\pi_*(r(\L+K_{\cX/\pr^1}))$, which we  identify with a trivialising section of $r(\L+K_{\cX/\pr^1})$, we define a $\Q$-divisor $\Delta$ on $\X$. As $$K_{\X} + \Delta = \L + K_{\pr^1}$$ is $\Q$-Cartier, we may define the multiplier ideal sheaf $\J((\X,\Delta,\|\L\|);c\X_0)$ using the above recipe.

\begin{definition} We define the \emph{log canonical threshold} of $\X_0$ on $(\X,\Delta,\|\L\|)$ to be $$\lct((\X,\Delta,\|\L\|);\X_0) = \sup \{c\in \R: \J((\X,\Delta,\|\L\|);c\X_0) = \scO_X\}.$$\end{definition}

This definition allows us to define the Ding invariant.

\begin{definition} The \emph{Ding invariant} of $(\X,\L)$ is defined as $$\Ding(\X,\L) =\lct((\X,\Delta,\|\L\|);\X_0) - 1- \frac{\langle \L^{n+1}\rangle}{(n+1)\langle (-K_X)^n\rangle}.$$
\end{definition}

\begin{definition} We say that $(X,-K_X)$ is \emph{Ding semistable} if for all big test configurations $(\X,\L)$ for $(X,-K_X)$ we have $$\Ding(\X,\L) \geq 0.$$

\end{definition}

All test configurations $(\X,\L)$ admit a natural morphism $$\X \backslash \{0\} \to X \times (\pr^1 \backslash \{0\})$$ induced by $p: Y \to X$. To define uniform Ding stability, we require in addition a norm, which is defined at first for \emph{dominant} test configurations, where this morphism extends over the central fibre.

\begin{definition} We say that a test configuration $(\X,\L)$ for $$p: (Y,L_Y) \to (X,L)$$ is a \emph{dominant} test configuration if it admits a surjective $\C^*$-equivariant morphism $$p: \X \to Y\times \pr^1 \to X\times \pr^1$$ compatible with the natural maps to $\pr^1$ and with $p: Y \to X$ and such that $K_{\X}$ is $\Q$-Cartier.\end{definition}

Note that given any test configuration for $p: (Y,L_Y) \to (X,L)$, one can (non-canonically) construct an associated dominant test configuration $$q: (\X',q^*\L) \to (\X,\L)$$ by taking an equivariant resolution of indeterminacy of the natural birational map $$Y\times \pr^1 \dashrightarrow \X, $$ and one can even take $\X'$ to be smooth. The following invariants generalise to the big case invariants introduced in \cite{BHJ1,uniform}.

\begin{definition} For a dominant test configuration $(\X,\L)$, setting $V=\langle (-K_X)^n\rangle = \langle L_Y^n\rangle$  we define 
\begin{enumerate}[(i)]
 \item the $J$-\emph{energy} $$J(\X,\L) = V^{-1}\langle \L\cdot L_Y^n\rangle-  (V(n+1))^{-1}\langle \L^{n+1}\rangle;$$
 \item the $I$-\emph{energy} $$I(\X,\L) = V^{-1}(-\langle \L^{n+1}\rangle - \langle \L^n\cdot L\rangle +\langle \L\cdot L_Y^n\rangle);$$
 \item the \emph{minimum norm} (or $I-J$-\emph{energy}) $$\|(\X,\L)\|_m = I(\X,\L) - J(\X,\L).$$
 \end{enumerate} For a general test configuration we define $J(\X,\L)$ to equal $J(\Y,\L_{\Y})$ for any associated dominant test configuration.
\end{definition}

It is straightforward to see that the for a general test configuration, these definitions---calculated on a resolution of indeterminacy---are independent of choice of such resolution. The following is proven in Section \ref{sec5} as Lemma \ref{NA-IJ-uniform}, using analytic techniques.

\begin{lemma}
There are uniform bounds \begin{align*}J(\X,\L) &\leq I(\X,\L) \leq (n+1)J(\X,\L), \\ n^{-1}J(\X,\L)&\leq \|(\X,\L)\|_m =  I(\X,\L)- J(\X,\L)\leq nJ(\X,\L). \end{align*}
\end{lemma}

\noindent Thus the ``norms'' defined by each of these quantities are equivalent.

\begin{definition} We say that $(X,-K_X)$ is  \emph{uniformly Ding stable} if there exists an $\epsilon>0$ such that for all  big  test configurations $(\X,\L)$ we have $$\Ding(\X,\L) \geq \epsilon \|(\X,\L)\|_m.$$
 \end{definition}
 
 The following will be useful in relating uniform Ding stability to its analytic counterpart.
 
 \begin{proposition}\label{prop:suff-dominant}$(X,-K_X)$ is uniformly Ding stable if and only if  there exists an $\epsilon>0$ such that for all smooth, dominant big test configurations $(\X,\L)$ we have $$\Ding(\X,\L) \geq \epsilon\|(\X,\L)\|_m.$$\end{proposition}
 
 \begin{proof}Take an associated dominant test configuration $(\X',q^*\L)$ with smooth total space; note that this test configuration remains big. Birational invariance of the positive intersection product means all that needs to be checked is that the $\lct$ term in Ding is unchanged. But since the line bundle on $\X'$ is the pullback $q^*\L$, for all divisorial valuations $E$ we have $$\ord_E\|\L\| = \ord_E\|q^*\L\|.$$ Similarly with $\Delta'$, the associated divisor on $\X'$, we have $$K_{\X'} + \Delta' = q^*(K_{\X} + \Delta)$$ by a direct calculation, which then implies the desired equality of log canonical thresholds. 
 \end{proof}
 
 \begin{remark}\label{rmk:kstab} For a general normal projective variety $X$ with a big line bundle $L$, one can also define uniform K-stability in an analogous way. In fact there are two possible definitions, which we expect to be equivalent, though this appears to be challenging. 
 
 To explain the first approach, firstly set  $$\mu(X,L) = \frac{-K_X\cdot \langle L^{n-1}\rangle}{\langle L^n\rangle}.$$ Defining the \emph{Donaldson--Futaki invariant} of a big test configuration as $$\DF(\X,\L) = \frac{n}{n+1}\mu(X,L)\langle \L^{n+1} \rangle + \langle \L^n\rangle  \cdot K_{\X/\pr^1},$$ we say that $(X,L)$ is \emph{uniformly K-stable} if there exists an $\epsilon>0$ such that for all big test configurations we have $$\DF(\X,\L) \geq \epsilon \|(\X,\L)\|_m.$$ This notion should be compared to related work of Li \cite{CL}, who defines an analogue of the Donaldson--Futaki invariant in this way when $\L$ is big but $L$ is \emph{ample}; we note birational equivalence classes of $(X,L)$ do not appear in his work (so in particular in his situation the general fibres of the test configuration are isomorphic to $(X,L)$ itself). 
 
The second approach uses instead intersection theory on the Riemann--Zariski space. Note that the canonical class defines a b-Weil divisor $K_{\underline X}$ on the Riemann--Zariski space $\underline X$ of $X$, where on each model $Y$ of $X$ the b-divisor $K_{\underline X}$ restricts to $K_Y$. Then one should be able to make sense of $$\mu(X,L) = \frac{-K_{\underline X}\cdot \langle L^{n-1}\rangle}{\langle L^n\rangle},$$ and  $$\DF(\X,\L) = \frac{n}{n+1}\mu(X,L)\langle \L^{n+1} \rangle + \langle \L^n\rangle  \cdot K_{\underline \X/\pr^1},$$ where the notation means that $\underline \X$ is the Riemann--Zariski space of $\X$ and these numbers are intersection numbers on the Riemann--Zariski space. The reason that this definition is not quite precise is that this is an intersection of $b$-divisors on the Riemann--Zariski space; see Dang--Favre for work in this direction for \emph{nef} $b$-divisors \cite{favre-dang}. The equivalence of these approaches is closely related to the entropy approximation conjecture; see work of Li when $L$ is ample \cite[Section 4]{CL-fujita}. 

 \end{remark}
 
\subsection{General singularity types}
 
When considering big line bundles, it is natural to consider test configurations $(\X,\L)$ with $\L$ big. While this is sufficient for our main results, we will prove our slope formulas for positive Deligne functionals in much greater generality. The goal of this section is thus to define more purely algebro-geometric counterparts of singularity types, which will be used to phrase our slope formulas. 

Let $\underline X$ be the Riemann--Zariski space of $X$, as defined in Section \ref{sect_23}, and let $L$ be a b-Cartier pseudoeffective divisor $L$ on $\underline X$. An \emph{enhanced $b$-divisor} $\mathbb D_{\underline X}$ in this space (following the terminology of \cite[Section 10.2]{xia-vb}) is an $\R$-divisor $$D_{Y} = \sum_E a_{E_{Y}} E_{Y}$$ on each $Y$ compatible under pushforward between models, where the sum is allowed to be countable. 

\begin{definition}\label{def:sing-type}
A \emph{$b$-singularity type} is an enhanced $b$-divisor $\mathbb D_{\underline X}$ on $\underline X$ such that
\begin{enumerate}[(i)]
\item $L - D_{Y}$ is nef in codimension one;
\item for models $\pi: Z\to Y$ the difference $D_Z-\pi^*D_Y$ is effective.
\end{enumerate}
\end{definition}
Note that on each model, a $b$-singularity type defines a class in the N\'eron--Severi group of $\Y$ by \cite[Remark 2.2]{xia-pp}, even though the sum may not be finite. The first condition, namely that on each model $Y$ we require $L - D_{Y}$ be nef in codimension one \cite{bou04}, is defined as follows. We say that a class $\alpha \in H^2(Y,\R)$ is \emph{modified K\"ahler} if there exists a birational model $\sigma: Y' \to X$ and a K\"ahler class $\alpha' \in  H^2(Y,\R)$ such that $\sigma_*\alpha' = \alpha$; this defines a cone in $H^2(Y,\R)$, and we say that $\beta$ is \emph{nef in codimension one} (or \emph{modified nef}) if $\beta$ lies in the closure of the modified nef cone. We refer to Boucksom for alternative characterisations \cite{bou04}.

by definition requires that there is a birational model $Y' \to Y$ and a nef 

\begin{example}

Let $L$ be a big line bundle. A canonical choice of $b$-singularity type associated to $L$ is obtained by setting $$D_{Y} =  \sum_E \ord_E\|L\|.$$ By \cite{bou04}, \cite[Corollary 5.21]{xia-okounkov}, for a metric $\phi$ with minimal singularities we have $$ \ord_E\|L\| = \nu_E(\phi),$$ where  $\nu_E(\phi)$ denotes the generic Lelong number of $\phi$ along $E$. That this defines a $b$-singularity type follows from \cite[Section 10.2]{xia-vb}. We write this singularity type as $[\phi]$.

\end{example}

\begin{example}\label{ex:xia}
Suppose $\phi$ is a singular metric on $L$. Then setting  $$D_{Y} = \sum_E \nu_E(\phi)$$ defines a $b$-singularity type by \cite[Section 10.2]{xia-vb} (see also \cite[Section 2.3]{xia-pp}). This is an invariant of the \emph{singularity type} of $\phi$, namely the equivalence class of singular metrics where $\phi,\psi$ are equivalent if there exists a constant $C$ such that $$\psi - C \leq \phi \leq \psi+C.$$ 

Conversely, recent work of Trusiani (independent to our own) shows that given a $b$-singularity type $\mathbb D_{\underline X}$ with  strictly positive volume (in a sense made precise below), there is a singular metric $\phi$ such that $\mathbb D_{\underline X}=[\phi]$ \cite{trusiani-ytd}. Trusiani proves this for $L$ ample, but the proof extends to the setting that $L$ is big (we thank Trusiani for advice on this point). So this notion of a $b$-singularity type is precisely the algebraic counterpart to the information obtained from  the singularity type of a singular metric.

\end{example}

We now consider an analogous notion on the test configuration. We begin by constructing an equivariant analogue of the Riemann--Zariski space of a test configuration; this is defined by $$\underline \X = \varprojlim_\Y \Y,$$ where $\Y \to \X$ is a  $\C^*$-\emph{equivariant} morphism. We also choose a b-Cartier divisor $\L$ on $\underline X$.

\begin{definition}
An \emph{invariant $b$-singularity type} is a $\C^*$-invariant enhanced $b$-divisor $\mathbb D_{\underline \X}$ on $\underline \X$ such that 
\begin{enumerate}[(i)]
\item $\L - D_{\Y}$ is nef in codimension one;
\item  for $\C^*$-equivariant models $\pi: \mathcal Z\to \Y$ the difference $D_{\mathcal Z}-\pi^*D_{\mathcal Y}$ is effective.

\end{enumerate}
We say that an invariant $b$-singularity type \emph{restricts to} $\mathbb D_{\underline X}$ if $\mathbb D_{\underline \X}$ restricts to $\mathbb D_{\underline X}$ on all fibres $\X_t \cong X$ for $t \neq 0$.

\end{definition}

Associated to an invariant $b$-singularity type on a test configuration, we can associate its volume $$\langle (\L-\mathbb D_{\underline \X})^{n+1}\rangle = \varprojlim_{\Y} \langle (\L- D_{\Y})^{n+1}\rangle,$$ where the right hand side being decreasing under nets implies the limit exists, as in Section \ref{sect_23}, by property $(ii)$ of the definition of an invariant $b$-singularity type (this observation also makes precise the use of volume in Example \ref{ex:xia}). Similarly one can make sense of log canonical thresholds involving $b$-singularity types by defining, for $E \subset \Y$ a divisorial valuation, $$\ord_E \mathbb D_{\underline \X} = \ord_E D_{\Y},$$ and following the valuative approach to log canonical thresholds presented in Section \ref{sec:dingstability}. Thus one can define the Ding invariant of a test configuration relative to an $b$-singularity type as $$\Ding_{\mathbb D_{\underline \X}}(\X,\L) =\lct((\X,\Delta,\mathbb D_{\underline \X});\X_0) - 1- \frac{\langle  (\L-\mathbb D_{\underline \X})^{n+1}\rangle}{(n+1)\langle (-K_X - \mathbb D_{\underline X})^n\rangle},$$ where we assume that $\langle (-K_X - \mathbb D_{\underline X})^n\rangle)$ is strictly positive. Thus invariant $b$-singularity types seem to be the right general setting to algebraically define the Ding invariant.

As in Example \ref{ex:xia}, an $S^1$-invariant singular metric $\Phi$ on $\L$ determines an invariant $b$-singularity type $\divisor_{\underline \X} \Phi$. The converse of this observation is encapsulated in the following definition.

\begin{definition}\label{def_realisable}Let $X$ be a  projective manifold and $L$ be a line bundle on $X$. Let $(\cX,\cL)$ be a test configuration for $(X,L)$. Let $\mathbb D_{\underline \X}$ be an invariant $b$-singularity type which restricts to $\mathbb D_{\underline X}$. We say that the test configuration $(\cX,\cL)$ is \textit{$(\mathbb D_{\underline X},\mathbb D_{\underline \X})$-realisable} if there exists an $S^1$-invariant  psh metric $\Phi$ on $\cL$ such that:
\begin{enumerate}[(i)]
	\item viewing the restriction of $\Phi$ to $\bbc^*$ as a function $z\mapsto \phi_z$, we have for all $z$ that $[\phi_z]=\mathbb D_{\underline X}$;
	\item $[\Phi]=\mathbb D_{\underline \X}$.
\end{enumerate} We then also say that $(\X,\L)$ is  \textit{$([\phi],[\Phi])$-realisable}. If $\psi$ and $\Psi$ are chosen to have minimal singularities, we simply say that the test configuration is \textit{realisable}. If $\Phi$ is a metric as described above, we say that it is a \textit{realising metric} for the given singularity types (or for $(\cX,\cL)$ when the singularity types are minimal).
\end{definition}
\section{Asymptotics of functionals and the main theorem}\label{sec5}

\subsection{Realisability of test configurations}\label{realisibility}

In order to prove slope formulae for the Ding functional, we must begin by associating rays of singular metrics to test configurations. We begin by  showing that all big test configurations are realisable in the sense of Definition \ref{def_realisable}. The following proof is a corrected version of a result in a previous version of this paper, where the assumption $\cL-L$ is effective was not assumed (where $L$ denotes the natural pullback to $\X$), and was suggested to us by Antonio Trusiani, whom we thank.
\begin{proposition}\label{prop_existencemin}
Let $L$ be a line bundle on $X$ and let $(\cX,\cL)$ be a test configuration for $(X,L)$, dominating the trivial test configuration via $\pi:\cX\to X\times\bbp^1$. If $L$ is pseudoeffective and if $\cL-L$ is effective, then $(\cX,\cL)$ is realisable. In that case, a realising metric can be chosen to be constant on the preimage of any open set that does not contain zero in its closure.
\end{proposition}
\begin{proof}
Let $\theta$ be a smooth representative of $c_1(L)$ and let $\phi\in\PSH(X,\theta)$ be a metric with minimal singularities. It naturally defines an invariant psh metric $\pi_\triv^*\phi$ on the trivial test configuration $(X\times \bbp^1,\pi_\triv^*\theta)$ which restricts to $\phi$ on any fibre. Then, $\pi^*\pi_X^*\phi$ is also psh on $\pi^*\pi_\triv^*\theta$ and restricts to $\phi$ on the general fibre. 

We now write $\L-L =  \cD$ with $\cD$ a (necessarily vertical) effective $\bbq$-Cartier divisor, where again $L$ denotes the pullback to $\X$ and where effectivity follows from the definition of a big test configuration. Picking $k$ such that $k\D$ is Cartier and a defining section $s_{kD}$ for $k\D$, we then have that $\Phi:=\pi^*\pi_X^*\phi+k\mi\log|s_{k{\D}}|$ is $\cL$-psh (because $\cD$ is effective) and restricts to $\phi$ away from the central fibre, hence already satisfies $(i)$ in the definition of realisability; it also has minimal singularities away from the general fibre. Now, taking $V_\Theta$ to be the minimal singularity envelope on $\cL$,  there exists $C$ such that $V_\Theta\geq \Phi-C$. In particular, this also holds fibrewise, since $\Phi$ has fibrewise minimal singularities and $C$ is independent of the variable on the base, thereby showing that $V_\Theta$ is realising. A choice of an adequate constant $C$ shows that we can also take $\max(V_\Theta-Ct,\Phi)$ to be realising, and such that its restriction is constant (and equal to $\phi$) on the preimage of an open set away from the central fibre.
\end{proof}

\begin{remark}We note that in general the realisability assumption for metrics with minimal singularities is strong: it forces metrics with global minimal singularities to also have fibrewise minimal singularities. Indeed, assume $\Phi$ is realising and $\Psi$ simply has global minimal singularities on $L$. Then there exists $C$ such that $|\Psi-\Phi|\leq C$; in particular, for all $t$, $|\phi_t-\psi_t|\leq C$, hence $\psi_t$ also has minimal singularities on $L$. The same argument also holds for arbitrary singularity types as in Definition \ref{def_realisable}. On the level of sections, this means that one obtains a strong comparison property between relative and global sections of $\cL$, akin to what always holds in the ample case, as in e.g. \cite[Lemma A.8]{bfjsolution}, and implies a similar orthogonality property to \cite[Lemma A.5]{bfjsolution}.
\end{remark}

\subsection{Slope formula for positive Deligne pairings of metrics with minimal singularities}

In this section, we prove a general slope formula for realisable (big) test configurations. Let $L_0,\dots,L_n$ be big line bundles on $X$ and, for each $i=0,\dots,n$, we define $\alpha_i:=c_1(L_i)$. We fix smooth representatives $\theta_i$ in each $\alpha_i$ and, for each $i$, a $\theta_i$-psh function $\psi_i$ with minimal singularities.

Let $(\cX_i,\cL_i)$ be dominant test configurations for each $(X,L_i)$ with $\L_i - L_i$ effective; we can and do assume that $\cX_i=\cX$ for each $i$, with $\cX$ smooth. We also assume that the general fibres $\X_t$ for $t\neq 0$ are isomorphic to $X$; the general situation is described in Corollary \ref{coro_generalcase}, where the results are as one would expect.  We let $\Theta_i$ be an $S^1$-invariant smooth representative of $c_1(\cL_i)$ which restricts to $\theta_i$ on the  fibre $\X_1 = X$.

Let $\Phi_i$ be a realising function for $\PSH(\cX,\Theta_i)$ in the sense of Definition \ref{def_realisable}, which exists by Proposition \ref{prop_existencemin}. This allows us to define a function
$$\llangle\Phi_0,\ldots,\Phi_n \rrangle: \bbc^* \to \R$$
(using abusive notation, since it depends on the choice of the reference functions $(\psi_i)_i$,) defined as
$$\llangle\Phi_0,\ldots,\Phi_n \rrangle(z)=\llangle \phi_{0,z},\dots,\phi_{n,z}\rrangle-\llangle \psi_0,\dots,\psi_n\rrangle.$$
By the definition of realisability, this is a locally bounded function on $\bbc^*$, and therefore we can consider its $dd^c$ in the sense of Bedford--Taylor \cite{bt}. Consider the restriction of the mixed Monge-Amp\`ere measure $$\langle \Theta_0+\ddc \Phi_0,\ldots,  \Theta_n + \ddc \Phi_n\rangle$$ to $X \times \bbc^* \subset\X$. The next main goal is to produce a potential for the pushforward of this measure.

To that end, we need to compare fibrewise and global Monge--Ampère operators. This will require us to consider \textit{restrictions} of currents on fibres of $\cX$ away from zero. Restrictions of currents are not well-defined in general, as one cannot pull back arbitrary currents. However, in the case that the currents are closed $(1,1)$-forms (in particular, if they arise as Monge--Ampère measures \cite[Theorem 1.8]{begz}) on submersions, one can use \textit{slicing theory} of currents, as explained in Demailly \cite[Section 1.C, p.13]{JPD}, which contains  further details on the general theory of slicing which we employ. Let us explain how this works in our particular case. A closed $(k,k)$-form $T$ on $\cX$ is in particular a \textit{locally flat} current, i.e. over each open set $U\subset \cX$ is of the form
$$T(f)=\int_U \langle \xi,f\rangle \,d\Omega-\int_U \langle \eta,df\rangle \,d\Omega,$$
where $\Omega$ is any choice of a smooth volume form on $\cX$, $f$ is any $(k,k)$-form, and $\xi$, $\eta$ are respectively a $(k,k)$ and a $(k+1,k+1)$-vector field, both with $L^1_{\mathrm{loc}}$ coefficients (with respect to $\Omega|_U$). 

Let us now assume that $U$ is an open set of the form $U=\pi\mi(V)$ for $V$ an open set in the base. Then, up to restricting to a smaller open set on the base if necessary, we can assume we are in the product case $U=U'\times V$, with $U'$ an open set in $\bbc^n$ and $V$ an open set in $\bbc$, and with $\Omega$ the usual volume form $i^{n+1} dx_1\wedge d\overline{x_1}\wedge\dots\wedge dx_n\wedge d\overline{x_n}\wedge dz\wedge d\bar z$ on $\bbc^n\times\bbc$. Note that the integrands $\langle \xi,f\rangle$ and $\langle \eta,df\rangle$ are both $L^1$ with respect to $\Omega$, by the flat decomposition above. It follows from Fubini's Theorem that the restriction of those integrands to a fibre $U'\times\{z\}$ is well-defined and integrable for $idz\wedge d\bar z$-almost all $z\in V$ (i.e. the set for which this does not produce the restriction is measure zero with respect to $idz\wedge d\bar z$). We then define the \textit{restriction} $T_z$ of $T$ to a fibre $X_z$ by restricting to $X_z$ the integrands in the above expression for $T$.

We may now prove the desired result:

\begin{proposition}\label{prop_curvature}
Using the notation above, fix for each $i$ a realising metric $\Phi_i$ for $\cL_i$. Then we have an equality of currents on $\bbc^*$
$$\int_{X\times\C^*/\bbc^*}\langle \Theta_0+\ddc \Phi_0\wedge \ldots \wedge\Theta_n + \ddc \Phi_n\rangle = \ddc \llangle\Phi_0,\ldots,\Phi_n \rrangle.$$
\end{proposition}
\begin{proof}
It suffices to show that for all functions $f$ of compact support on $\bbc^*$ that $$\int_{\bbc^*} f \int_{X\times {\bbc^*}/{\bbc^*}}\langle \Theta_0+\ddc \Phi_0\wedge\ldots\wedge  \Theta_n + \ddc \Phi_n\rangle = \int_{\bbc^*} f \ddc \llangle\Phi_0,\ldots,\Phi_n \rrangle,$$ where by definition we have  $$\int_{\bbc^*} f \int_{X\times {\bbc^*}/{\bbc^*}}\langle \Theta_0+\ddc \Phi_0\wedge\ldots\wedge  \Theta_n + \ddc \Phi_n\rangle = \int_{X\times {\bbc^*}}  f \langle \Theta_0+\ddc \Phi_0\wedge\ldots\wedge  \Theta_n + \ddc \Phi_n\rangle.$$

The definition of the current $\ddc \llangle\Phi_0,\ldots,\Phi_n \rrangle$ implies \begin{align*} \int_{\bbc^*} f \ddc \llangle\Phi_0,\ldots,\Phi_n \rrangle &=   \int_{\bbc^*}  \llangle\Phi_0,\ldots,\Phi_n \rrangle \ddc f, \\ &= \int_{\bbc^*} (\llangle \phi_{0,z},\ldots,\phi_{n,z}\rrangle-\llangle \psi_0,\ldots,\psi_n\rrangle) \ddc f.  \end{align*} This is an integral of positive Deligne functionals defined fibrewise, and we claim that we can write $$\int_{\bbc^*} (\llangle \phi_{0,z},\ldots,\phi_{n,z}\rrangle-\llangle \psi_0,\ldots,\psi_n\rrangle) \ddc f =  \int_{X\times {\bbc^*}} \ddc f\wedge\sum_i T_i',$$
where
$$T_{i}'=(\Phi_i-\Psi_i)\langle (\Theta_{0}+dd^c\Psi_0)\wedge \dots\wedge (\Theta_{i-1}+dd^c \Psi_{i-1})\wedge(\Theta_{i+1}+dd^c\Phi_{i+1})\wedge\dots\wedge(\Theta_{n}+dd^c\Phi_{n})\rangle
$$ is such that the mixed Monge--Amp\`ere operators are defined globally. Here each $\Psi_i$ is a realising psh function on $\cL_i$ which restricts to $\psi_i$ on the fibres above the support of $f$, as one can obtain from Proposition \ref{prop_existencemin}. This requires a comparison of fibrewise and global Monge-Amp\`ere operators, which we prove using slicing theory in Lemma \ref{slicing}.

More precisely, what will be proven in Lemma \ref{slicing} is that that for almost every $z \in \C^*$ in the support of $f$ (i.e. away from zero) one has the following restriction property:
\begin{equation}\label{restriction-property}T_{i,z}=T'_i|_{X_z}\end{equation}
with
$$T_{i,z}=(\phi_{i,z}-\psi_i)\langle (\theta_{0}+dd^c \psi_0)\wedge \ldots\wedge (\theta_{i-1}+dd^c \psi_{i-1})\wedge(\theta_{i+1}+dd^c\phi_{i+1,z})\wedge\ldots\wedge(\theta_{n}+dd^c\phi_{n,z})\rangle,$$ so that $T'_i$ is defined using global Monge--Ampère operators
 whereas $T_{i,z}$ is instead defined  using fibrewise Monge--Ampère operators. Here,
$T'_i|_{X_z}$
is the restriction of the current $T'_i$ to the fibre $X_z$ as explained immediately before the statement of the proposition; we recall that it is well-defined for almost every $z$ on the base. Note that with this (fibrewise) definition $$\int_{ {\bbc^*}} \ddc f \int_{X\times\C^*/\C^*}\sum_i T_{i,z} = \int_{\bbc^*} (\llangle \phi_{0,z},\ldots,\phi_{n,z}\rrangle-\llangle \psi_0,\ldots,\psi_n\rrangle) \ddc f.$$

Assuming this restriction property, by the fibre-integral formula in slicing theory  \cite[Equation (1.22)]{JPD}  we obtain
$$\int_{ {\bbc^*}} \ddc f \int_{X\times\C^*/\C^*}\sum_i T_i=\int_{X\times\C^*} \ddc f\wedge \sum_i T'_i,$$ where we have identified $f$ with its pullback to $X\times\C^* \subset \X$, which still has compact support. Since $f$ is bounded on $\cX$, and the $\Phi_i$ as well as the $\Psi_i$ have minimal singularities, we can apply non-pluripolar integration by parts on the (compact) manifold $\cX$ to obtain
\begin{align*}
&=\int_{\cX} \ddc f\wedge \sum_i T'_i\\
&=\int_{\cX} f(dd^c\Phi_i-dd^c \Psi_i)\\
&\wedge  \sum_i \langle (\Theta_{0}+dd^c\Psi_0)\wedge \ldots\wedge (\Theta_{i-1}+dd^c \Psi_{i-1})\wedge(\Theta_{i+1}+dd^c\Phi_{i+1})\wedge\ldots\wedge(\Theta_{n}+dd^c\Phi_{n})\rangle\\
&=\int_{X\times\bbc^*} f(dd^c\phi_{i,z}-dd^c \psi_i)\\
&\wedge \sum_i \langle (\theta_{0}+dd^c\psi_0)\wedge \ldots\wedge (\theta_{i-1}+dd^c \psi_{i-1})\wedge(\theta_{i+1}+dd^c\phi_{i+1,z})\wedge\ldots\wedge(\theta_{n}+dd^c\phi_{n,z})\rangle.
\end{align*}
Using multilinearity of the mixed Monge-Amp\`ere operator and the fact that $$\langle \wedge^{n+1}(\theta_i+\ddc \psi_i)\rangle = 0,$$ we see that 
\begin{align*} &(dd^c\phi_{i,z}-dd^c \psi_i)\wedge \sum_i \langle (\theta_{0}+dd^c\psi_0)\wedge \ldots\wedge (\theta_{i-1}+dd^c \psi_{i-1})\wedge(\theta_{i+1}+dd^c\phi_{i+1,z})\wedge\ldots\wedge(\theta_{n}+dd^c\phi_{n,z})\rangle \\ &=  \langle \theta_0 + \ddc \phi_{0,z},\ldots,  \theta_n + \ddc \phi_{n,z}\rangle.
\end{align*}Since in turn  by definition this quantity satisfies $$\int_{X\times\C^*/\C^*}f\langle \theta_0 + \ddc \phi_{0,z}\wedge\ldots\wedge  \theta_n + \ddc \phi_{n,z} \rangle= \int_{\bbc^*} f \int_{X\times {\bbc^*}/{\bbc^*}}\langle \Theta_0+\ddc \Phi_0\wedge\ldots \wedge \Theta_n + \ddc \Phi_n\rangle, $$ the proof is complete.
\end{proof}

We next prove the restriction property used in the proof, for which we continue employing the same notation.

\begin{lemma}\label{slicing} For almost every $z\in\C^*$, the  restriction of $T'_i$ is $T_{i,z}$: $$T_{i,z}=T'_i|_{X_z}.$$ \end{lemma}

\begin{proof}
As explained in \cite[p.171]{demBook}, the slice $T'_i|_{X_z}$ is defined for almost every  $z\in\bbc^*$, and can be obtained as a limit of genuine restrictions $\lim_{\varepsilon\to 0}T'_{i,\varepsilon}|_{X_z}$, where $T'_{i,\varepsilon}\to T'_i$ is a sequence of regularisations of $T'_i$ \textit{via} mollifiers.

We claim that this approximation process applies for an arbitrary choice of smooth forms converging to $T'_i$ in $L^1_{\mathrm{loc}}$. Assuming this is true for the moment, we may pick a smooth approximation $\Phi_{i,\varepsilon}$ of each $\Phi_i$, so that defining $T_{i,\varepsilon,z}$ and $T'_{i,\varepsilon}$ as in the previous proposition, replacing the $\Phi_i$ with $\Phi_{i,\varepsilon}$, one has for all $z$ on the base
$$T_{i,\varepsilon,z}=T'_{i,\varepsilon}|_{X_z}$$
as pull-backs commute with exterior derivatives (in the smooth case, as is the case here), and on taking the limit in $\varepsilon$, it follows from the first paragraph that the desired comparison property
$$T_{i,z}=T'_i|_{X_z}$$
holds for almost every $z$ on the base.

We now prove our previous claim. For simplicity, let us write $T_i=T$, and choose smooth approximations $S_\varepsilon,T_\varepsilon\to T$ in $L^1_{\mathrm{loc}}$. Decomposing $T$ locally on $U\times X$ as explained before the proposition, we write $T=\xi+\eta V$, and $S_\varepsilon=\xi_{S_\varepsilon}+d\eta {S_\varepsilon}$, $T_\varepsilon=\xi_{T_\varepsilon}+d\eta_{T_\varepsilon}$. Then, $T|_{X_z}$ is defined for almost every $z$ via
$$T|_{X_z}(f)=\int_{X_z}(\langle \xi,f\rangle-\langle \eta,df\rangle)|_{X_z} \,d\Omega|_{X_z}$$
for $f\in \mathrm{Im}(C^\infty(U\times X)\to C^\infty(X_z))$, and $\Omega$ a given smooth form on $U\times X$. By convergence of $S_\varepsilon$ and $T_\varepsilon$, it follows that $\xi_{S_\varepsilon},\xi_{T_\varepsilon}\to \xi$ and likewise for $\eta$, hence
$$S_\varepsilon|_{X_z}(f),T_\varepsilon|_{X_z}(f)\to T|_{X_z}(f),$$
concluding the proof.
\end{proof}

Assuming the functions $\Phi_i$ are $S^1$-invariant, we can in fact identify their restriction to $X\times\bbd^*$ with a function
$$t\mapsto \phi_{i,t}:=\phi_{i,-\log|z|}$$
for $t\in[0,\infty)$. Similarly, the function $\llangle \Phi_0,\dots,\Phi_n\rrangle$ can be seen as a function on $[0,\infty)$. We then obtain a slope formula as desired.
\begin{theorem}\label{thm-deligne-slope}
Under the assumptions of Proposition \ref{prop_curvature}, there is a slope formula $$\langle \cL_0\cdot\ldots\cdot\cL_n\rangle = \lim_{t\to\infty}  \frac{\llangle\Phi_0,\ldots,\Phi_n \rrangle(t)}{t}.$$
\end{theorem}
\begin{proof}
Since the currents $\Theta_i+\ddc \Phi_i$ have minimal singularities, we know that $$\langle \cL_0\cdot\ldots\cdot\cL_n\rangle = \int_{\X} \langle (\Theta_0+\ddc \Phi_0) \wedge\dots\wedge (\Theta_n+\ddc \Phi_n)\rangle .$$ The mixed Monge--Amp\'ere measure puts no mass on pluripolar sets, so we have \begin{align*}\int_{\X }  \langle (\Theta_0+\ddc \Phi_0) \wedge\dots\wedge (\Theta_n+\ddc \Phi_n)\rangle &= \int_{\X_{\C^*}}  \langle (\Theta_0+\ddc \Phi_0) \wedge\dots\wedge (\Theta_n+\ddc \Phi_n)\rangle, \\ &= \int_{X\times \C^*} \langle (\Theta_0+\ddc \Phi_0) \wedge\dots\wedge (\Theta_n+\ddc \Phi_n)\rangle,  \\ &= \int_{\C^*} \ddc \llangle\Phi_0,\ldots,\Phi_n \rrangle,  \\ &=   \lim_{t\to\infty}  \frac{\llangle\Phi_0,\ldots,\Phi_n \rrangle(t)}{t},\end{align*} where the second-to-last equality follows from Proposition \ref{prop_curvature} and the last equality follows from a calculation in polar coordinates as in \cite[Lemma 2.6]{bermankpoly}. The result follows.\end{proof}

We end by briefly discussing the situation when each $\Phi_i \in \PSH(\X,\Theta_i)$ is not necessarily taken to have minimal singularities.  Here the constructions go through in a very  similar manner, and using the same notation along with the results of Section \ref{sect_23} produces the following corollary:

\begin{corollary}\label{coro_generalcase} For an $S^1$-invariant metric $\Phi \in \PSH(\X,\Theta)$ which has fixed singularity type $[\phi_z]=[\phi]$ away from the central fibre and such that $(\cX,\cL)$ is $([\phi],[\Phi])$-realisable, we have $$\lim_{k\to\infty} \frac{n!\dim H^0(\X,k\L\otimes\J(k\Phi))}{k^{n+1}} = \lim_{t\to\infty}  \frac{E_{[\phi_0]}\langle\phi_t\rangle}{t},$$
where the Monge--Ampère energy on the right-hand side is the energy relative to the singularity type $[\phi_0]$. 

More generally, for $S^1$-invariant metrics $\Phi_i \in \PSH(\X,\Theta_i)$ with fixed singularity types $[\phi_{i,z}]=[\phi_i]$ away from the central fibre and such that $\cL_i$ is for each $i$ a $([\psi_i],[\Psi_i])$-realisable test configuration for $L_i$, we have $$ (\L_0 - \divisor_{\underline X} \Phi_0) \cdot \ldots \cdot (\L_n - \divisor_{\underline X} \Phi_n) = \lim_{t\to\infty}  \frac{\llangle\Phi_0,\ldots,\Phi_n \rrangle_{[\phi_0],\dots,[\phi_n]}(t)}{t}.$$
where the right-hand side is a relative positive Deligne pairing obtained by polarising the definition of the Monge--Ampère energy relative to a singularity type.
\end{corollary}

The latter slope is an invariant of the $b$-singularity type determined by $\Phi$.

\subsection{Finite-energy spaces and birational equivalence}

In order to study Ding-stability for dominant test configurations, we will need to consider psh functions living in energy spaces over different manifolds. The following result explains how to relate such energy classes.

\begin{theorem}[{\cite[Theorem 3.5]{DNFT}, \cite[Remark 2.4]{dinezzastability}}]\label{thm_pullbackenergy}
Let $\pi:Y\to X$ be a birational morphism. Let $L_Y$ be a big line bundle on $Y$, and let $\theta_Y\in c_1(L_Y)$ be a smooth representative. Then,
$$\cE^1(Y,\theta_Y)=\cE^1(X,\pi_*\theta_Y)$$
if and only if
$$\vol(\theta_Y)=\vol(\pi_*\theta_Y)$$
and
$$\{\theta_Y\}=\pi^*\pi_*\{\theta_Y\}+E.$$
and this identification preserves the property of having minimal singularities. If $\pi:Y\to X$ is a blowup with smooth centre and exceptional $\bbq$-divisor $E$, and $s_{cE}$ is a defining section for $cE$ with $c$ such that $cE$ is Cartier, then $\phi\in\PSH_{\mathrm{min}}(X,\pi_*\theta_Y)$ is identified with
$$\pi^*\phi + c\mi\log|s_{cE}|.$$
\end{theorem}

\begin{remark}
Through this identification, it is clear that the proof of Theorem \ref{thm-deligne-slope} goes through in the same manner, producing slope formulas for Deligne functionals on general test configurations whose central fibre is given by a pair $(Y,L_Y)$ birationally equivalent to $(X,L)$.
\end{remark}

\subsection{Slope formula for the Ding energy}

We assume from now on that $-K_X$ is big, and that $X$ is klt. We set $L:=-K_X$ for notational convenience. We then consider $\pi:(Y,L_Y)\to (X,L)$ a birational equivalence, and $$\Pi:(\cY,\cL_\cY)\to (X\times\bbp^1,L\times\bbp^1)$$ a big and dominant test configuration for $\pi: Y \to X$ with smooth total space. Fixing a smooth representative $\Theta_Y$ for $c_1(\cL_\cY)$, let $\Phi_\cY$ be a $S^1$-invariant psh metric in $\PSH(\cY,\Theta_\cY)$  with relative minimal singularities away from the central fibre. We now prove a slope formula for the Ding energy along $\Phi_\cY$. The strategy is similar to Berman's work in the ample case \cite{bermankpoly}. Suppose $r>0$ is such that $r\L_\cY$ is Cartier. We will consider throughout the $\Q$-line bundle $$r\mi(\pi_{\bbp^1}\circ\Pi)_*(r(\L_\Y+K_{\cY/\pr^1}))$$ (using Lemma \ref{lem_pushfwd}); but to ease notation we will assume throughout that $r=1$. The proofs in the general case only cause notational difficulty.

\bigskip Let $s$ be a trivialising section of $(\pi_{\bbp^1}\circ\Pi)_*(\L_\Y+K_{\cY/\pr^1})$ over a neighbourhood $U$ of $0$ (which we can assume to be a disc) in $\bbp^1$, and let $\cY_U$ be the restriction of $\cY$ to this neighbourhood. We may thus realise $s$ as a section of $(\L_\Y+K_{\cY/\pr^1})$ over $\cY_U$, and we may assume that its restriction to a fibre $s_t=s|_{\cY_t}$ is constant in $t$ when identified as a section of $L_Y+K_Y$ through the identification $L_Y+K_Y\cong \cL_\cY+K_{\cY/\bbp^1}$. As in Section \ref{sect_canmeasures}, since for $t\neq 0$ we obtain a section $s_t$ of the adjoint bundle $L_Y+K_Y$, we can associate to $s_t$ and $\phi_{\cY,t}$ a measure
$$\mu_{\cY,t}:=|s|^2e^{-\phi_{\cY,t}}.$$
We then define
$$v(t):=\int_{Y}\mu_{\cY,t}.$$
Note that $\phi_{\cY,t}$ can be identified with a psh metric $\phi_t$ on $-K_X$ with minimal singularities in $\PSH(X,\theta)=\PSH(X,\pi_*\theta_Y)$ via Theorem \ref{thm_pullbackenergy}. It defines again as in Section \ref{sect_canmeasures} a canonical measure $\mu_t:=e^{-\phi_t}$, so that $$-\log\int_X \mu_t= L\langle\phi_t\rangle.$$
\begin{lemma}\label{lem_invarianceding}
For all $t$, we have
$$v(t)=\int_X \mu_t.$$
\end{lemma}
\begin{proof}
As explained in Theorem \ref{thm_pullbackenergy}, we have $\phi_{\cY,t}=\pi^*\phi_t+c\mi \log|s_{cE}|$, with $E$ the exceptional divisor of $\pi$, $c$ such that $cE$ is Cartier, and $s_{cE}$ a defining section for $cE$. In particular, by construction, the measures $\mu_{\cY,t}$ and $\mu_t$ coincide on the set where $\pi$ is an isomorphism, and are extended by zero away from this set, proving the lemma.
\end{proof}

\begin{theorem}[P\u{a}un--Takayama]\label{thm_pshding}The function $v(t)$ is subharmonic on a neighbourhood of $0$ in $\pr^1$.\end{theorem}

\begin{proof} This follows directly from a general result of P\u{a}un--Takayama \cite{pauntakayama}. To see this, one views $v(t)$ as the potential for the natural $L^2$-metric on $(\pi_{\bbp^1}\circ\Pi)_*(\L_\Y+K_{\cY/\pr^1})$, to which \cite[Theorem 1.1]{pauntakayama} implies the desired subharmonicity. \end{proof}

\begin{remark} As in Berman \cite[Lemma 3.2]{bermankpoly}, This result also follows from earlier work of Berndtsson--P\u{a}un \cite{bpaun}, arguing instead similarly to Berman \cite[Lemma 3.2]{bermankpoly} using a version of the Ohsawa--Takegoshi extension theorem. \end{remark}

Since $v$ is subharmonic, it has a well-defined Lelong number at zero. By construction, $v$ is also $S^1$-invariant, so that its Lelong number at $0$ is then computed by the slope at infinity
$$\nu_0(v)=\lim_{t\to\infty}\frac{\log(v(t))}{t}.$$
We claim that this slope at infinity is in turn computed by a complex singularity exponent associated to a canonical measure, in the following sense.
\begin{definition}
As in the previous discussion, the data of a section of $(\L_\Y+K_{\cY/\pr^1})$ over $\cY_U$ yields a measure $|s|^2e^{-\Phi_\cY}$ away from the polar locus of $\Phi_\cY$, which we extend by zero there; this measure extends the family $t\mapsto \mu_{\cY,t}$ considered before, over the central fibre. Given any continuous $(n+1,n+1)$-current $\Phi$ seen as a measure $\mu_\Phi$ on $\cX$, let us denote $e^{-\gamma_\cY}$ the Radon--Nikodym derivative of $|s|^2e^{-\Phi_\cY}$ with respect to $\mu_\Phi$. We then define the \textit{complex singularity exponent}
$$c_{\cY_0}(\cY,\Delta_\cY,\Phi_\cY) = \sup\{c \in \R, \exists\textrm{ a neighbourhood } U\textrm{ of }\cY_0 \textrm{ with } e^{-c\gamma}\in L^1_{\mathrm{loc}}(U, \mu_\Phi)\}$$
with $\Delta_\cY:=-\cL_\cY-K_{\cY/\bbp^1}$.
\end{definition}
\begin{theorem}\label{thm_slopeding1}
There is an equality 
$$\lim_{t\to\infty}\frac{v(t)}{t}=1 - c_{\cY_0}(\cY,\Delta_\cY,\Phi_\cY).$$
\end{theorem}

\begin{proof}
A general computation as in \cite[Proposition 3.8]{bermankpoly} shows that
$$\lim_{t\to\infty}\frac{v(t)}{t}=\inf \left \{ c\in \R: \int_U e^{-v(z)}e^{(1-c)\log|z|^2}idz\wedge d\bar z<\infty\right\}.$$
while viewing $v$ as a potential for the $L^2$-metric as in Theorem \ref{thm_pshding}, we see that for any function $f$ on $U$ we have 
$$\int_{\Y_U}(\pi_U\circ\Pi)^*f\,e^{-\gamma_\cY}\mu_\Phi=\int_{\Y_U}(\pi_U\circ\Pi)^*f\,|s|^2e^{-\Phi_\cY}= \int_U e^{-v(z)}f(z) idz \wedge d\bar z.$$
Setting $f(z) = \exp{(1-c)\log|z|^2}$ we then have
$$\int_{\Y_U}e^{(1-c)\log|z|^2}e^{-\gamma_\cY}\mu_\Phi= \int_U e^{-v(z)}e^{(1-c)\log|z|^2} idz \wedge d\bar z.$$
But 
$$\inf\left\{c\in\bbr,\,\int_{\Y_U}e^{(1-c)\log|z|^2}e^{-\gamma_\cY}\mu_\Phi<\infty\right\}=1-c_{\cY_0}(\cY,\Delta_\cY,\Phi_\cY),$$
concluding the proof.
\end{proof}

\begin{remark} We emphasise that we do not need $\Phi_{\Y}$ to have global minimal singularities for this result to hold, so along with Corollary \ref{coro_generalcase} we obtain a very general slope formula for the Ding functional.\end{remark}

\begin{corollary}\label{coro_slopeding2}
If $\Phi_\cY$ has global minimal singularities, then
$$\lim_{t\to\infty}\frac{v(t)}{t}=1 - \mathrm{lct}((\cY,\Delta_\cY,\|\cL_\cY\|);\cY_0).$$
\end{corollary}
\begin{proof}

The valuative criterion for integrability \cite[Corollary 10.14]{boucksom-l2} states that, given psh functions $\phi, \psi$ on a complex manifold $X$ such that $\psi$ has analytic singularities, then around a point $x \in X$ we have the local integrability condition $\exp(\psi - \phi) \in L^2_{loc}$ around $x$ if and only if there exists an $\epsilon>0$ such that $$\nu_E(\psi) \geq (1+\epsilon) \nu_E(\phi) - A_X(E)$$ for all prime divisors $E$ over $X$ whose image in $X$ contains the given point $x \in X$. Translating this into our valuative definition of multiplier ideal sheaves given as Equation \eqref{valuative-MIS} (where $\psi$ corresponds to the negative part $N$ of the divisor used there, which thus has analytic singularities), and hence into our valuative definition of the log canonical threshold, this precisely proves 
$$\mathrm{lct}((\cY,\Delta_\cY,\|\cL_\cY\|);\cY_0)=c_{\cY_0}(\cY,\Delta_\cY,\Phi_\cY),$$ which implies the result by Theorem \ref{thm_slopeding1}.
\end{proof}

\subsection{The main theorem}

We now bring together the various results. We let $(X,-K_X)$ be such that $-K_X$ is big, and we fix a reference metric on $-K_X$ with minimal singularities. We assume $\pi:Y\to X$ is a birational equivalence and $\Pi:(\X,\L)\to (X \times \pr^1,-K_X)$ is a dominant big test configuration for $\pi$ with smooth total space, and let $\Phi$ be a singular $S^1$-invariant psh metric on $\L$ which has global and relative minimal singularities. We set $D(t):=D\langle \phi_t\rangle$, $J(t):=J\langle\phi_t\rangle$, $E(t):=E\langle\phi_t\rangle$.

\begin{corollary}\label{coro_slopes} The slope formulas $$\lim_{t \to \infty} \frac{D(t)}{t} = \Ding(\X,\L)$$ and $$\lim_{t \to \infty} \frac{J(t)}{t} = J(\X,\L)$$ hold. \end{corollary}

\begin{proof} 
By Theorem \ref{thm-deligne-slope} and the fact that metrics with minimal singularities compute positive intersection numbers, we have $$\lim_{t \to \infty} \frac{E(t)}{t} = \frac{\langle \L^{n+1}\rangle}{n+1},\,\qquad\lim_{t \to \infty} \frac{J(t)}{t} = J(\X,\L).$$
In turn, Corollary \ref{coro_slopeding2} provides $$\lim_{t\to\infty}\frac{v(t)}{t}=1 - \mathrm{lct}((\cY,\Delta_\cY,\|\cL_\cY\|);\cY_0)$$
in the notation of the previous section; adding these proves the slope formula for the Ding functional.
\end{proof}

\begin{corollary} The following direction holds in the Yau--Tian--Donaldson conjecture:

\begin{enumerate}[(i)]
\item   if $X$ admits a K\"ahler--Einstein metric, it is Ding semistable;
\item  if $X$ admits a unique K\"ahler--Einstein metric, it is uniformly Ding stable.
\end{enumerate}
 \end{corollary}
 \begin{proof}
We prove (ii). Let $\theta\in c_1(-K_X)$ be a smooth representative. By Theorem \ref{thm_dingcoercive}, if $X$ admits a unique Kähler--Einstein metric, then $D$ is coercive on $\cE^1(X,\theta)$. Given $\pi:Y\to X$ a birational equivalence and $\theta_Y$ a big cohomology class on $Y$ such that $\pi_*\theta_Y=\theta$, Lemma \ref{lem_invarianceding} and Theorem \ref{thm_pullbackenergy} imply that $D$ is also coercive on $\cE^1(Y,\theta_Y)$. We now pick a dominant big test configuration with smooth total space (which is sufficient by Proposition \ref{prop:suff-dominant}) $(\cX,\cL)$ for $\pi$, and a $S^1$-invariant metric $\Phi=\{\phi_t\}_t$ on $\cL$ with minimal singularities. By coercivity of $D$ on $\cE^1(Y,\theta_Y)$ and using Theorem \ref{thm_pullbackenergy} (since the Monge--Ampère measures used in defining $J$ and $d_1$ will coincide on birational equivalences away from a pluripolar set) there exists $\delta>0$, $C>0$ (independent of $(\cX,\cL)$) such that for all $t$
$$D\langle\phi_t\rangle\geq \delta J\langle\phi_t\rangle-C,$$
which yields uniform Ding stability upon taking slopes and using Corollary \ref{coro_slopes}. The statement (i) is proven similarly using Theorem \ref{thm_dingmin} (which implies the Ding functional is bounded below in this situation).
 \end{proof}
 
 In precisely the same way we obtain a slope formula for the Ding functional for arbitrary realisible invariant $b$-singularity types, which proves a slope formula for the Ding functional in the most general setting possible.
 
 \begin{corollary}
 
 For a $([\phi],[\Phi])$-realisible test configuration $(\X,\L)$ we have $$\lim_{t \to \infty} \frac{D(t)}{t} = \Ding_{[\Phi]}(\X,\L).$$
 
 \end{corollary}
 
 We end by proving a promised property of the norms of a test configuration.
 
\begin{lemma}\label{NA-IJ-uniform}
For a big test configuration $(\X,\L)$, there are uniform bounds \begin{align*}J(\X,\L) &\leq I(\X,\L) \leq (n+1)J(\X,\L), \\ n^{-1}J(\X,\L)&\leq \|(\X,\L)\|_m = I(\X,\L)- J(\X,\L)\leq nJ(\X,\L). \end{align*}
\end{lemma}

\begin{proof} The main point is that each of the $I,J$ and $I-J$ energies can be viewed as Deligne functionals. That is, by Lemma \ref{IJ-uniform} we have analogous analytic uniform bounds for the $I,J$ and $I-J$ energies on $\E^1(X,\theta)$. Associating to $(\X,\L)$ a metric in $c_1(\L)$ with minimal singularities and using Theorem \ref{thm-deligne-slope}, we obtain analogous uniform bounds for their slopes, which are precisely $ I(\X,\L), J(\X,\L)$ and $I(\X,\L)- J(\X,\L) = \|(\X,\L)\|_m$ (with the latter equality holding by definition).
\end{proof}

\begin{remark}

If one takes a big test configuration $(\X,\L)$ and an $S^1$-invariant metric $\Omega \in c_1(\L)$ with minimal singularities then the one obtains a path $\phi_t$ with slope $$\lim_{t\to\infty}\frac{D\langle\phi_t\rangle}{t} =\Ding(\X,\L).$$ The functional $D$ is translation invariant in the sense that $D\langle\phi+c\rangle=D\langle\phi\rangle$ for all $c \in \R$ (as are  $I$ and $J$). The ``asymptotic'' analogue of this property asks that $\Ding(\X,\L)$ equals $\Ding(\X,\L+\scO_{\pr^1}(k))$ for all $k$; this property is usually also called translation invariance. In the big setting, the Ding invariant is \emph{not} translation invariant, even though the Ding functional is. The reason for this is that the sum of two metrics with minimal singularities may not have minimal singularities in general, which is the reason that the positive intersection product is not multilinear, while the mixed Monge--Amp\`ere operator is multilinear. Since the Ding invariant is computed as the slope of the Ding functional,  what \emph{is} true is that $$\Ding_{[\Omega+k\omega_{FS}]}(\X,\L+\scO_{\pr^1}(k))=\Ding_{[\Omega]}(\X,\L),$$ and so the obstruction to translation invariance lies in the equality of singularity types. In particular, translation invariance \emph{does} hold once the condition on singularity types is achieved, for example when both $\Omega$ and $\Omega+k\omega_{FS}$ have minimal singularities.

\end{remark}

\bibliographystyle{alpha}
\bibliography{bib}

\end{document}